\documentclass[a4paper,UKenglish,cleveref, autoref, thm-restate, numberwithinsect]{lipics-v2021}

\pdfoutput=1 
\hideLIPIcs  

\title{
Measuring Depth of Matroids} 

\author{Jakub Balabán}{Faculty of Informatics, Masaryk University, Brno, Czech Republic}{jakbal@mail.muni.cz}{https://orcid.org/0000-0002-2475-8938}{Brno Ph.D. Talent Scholarship Holder – Funded by the Brno City Municipality}

\author{Petr Hlin{\v{e}}n{\'{y}}}{Faculty of Informatics, Masaryk University, Brno, Czech Republic}{hlineny@fi.muni.cz}{https://orcid.org/0000-0003-2125-1514}{}

\author{Jan Jedelský}{Faculty of Informatics, Masaryk University, Brno, Czech Republic}{xjedelsk@fi.muni.cz}{https://orcid.org/0000-0001-9585-2553}{}

\author{Kristýna Pekárková}{AGH University of Kraków, Poland}{pekarkova@agh.edu.pl}{https://orcid.org/0000-0003-3539-6431}{}

\authorrunning{J. Balabán and P. Hlin{\v{e}}n{\'{y}} and J. Jedelský and K. Pekárková} 

\Copyright{Jakub Balabán and Petr Hlin{\v{e}}n{\'{y}} and Jan Jedelský and Kristýna Pekárková} 

\let\subjclassHeading\relax


\category{} 

\relatedversion{} 

\funding{The authors have been supported by the project 24-11098S of the Czech
Science Foundation.}

\nolinenumbers 

\EventEditors{John Q. Open and Joan R. Access}
\EventNoEds{2}
\EventLongTitle{42nd Conference on Very Important Topics (CVIT 2016)}
\EventShortTitle{CVIT 2016}
\EventAcronym{CVIT}
\EventYear{2016}
\EventDate{December 24--27, 2016}
\EventLocation{Little Whinging, United Kingdom}
\EventLogo{}
\SeriesVolume{42}
\ArticleNo{23}

\usepackage[utf8]{inputenc}
\usepackage[T1]{fontenc}
\usepackage{amsmath}
\usepackage{amssymb}
\usepackage{amsthm}
\usepackage{stmaryrd}
\usepackage{mathrsfs} 
\usepackage{mathtools}
\usepackage{thm-restate}

\usepackage{url}
\usepackage{todonotes}
\usepackage{mathrsfs} 
\usepackage{soul}
\usepackage{relsize}
\usepackage[all,defaultlines=3]{nowidow}
\usepackage[mathlines]{lineno}
\usepackage{multicol}
\usepackage{xspace}
\usepackage{lineno}
\usepackage{nicefrac}
\usepackage{wasysym}
\usepackage{textpos}

\newcommand{\FF}{\mathbb{F}}

\newcommand{\ZZ}{\mathbb{Z}}

\DeclareMathOperator{\cl}{cl}

\DeclareMathOperator{\hght}{ht}

\DeclareMathOperator{\mtw}{mtw}
\DeclareMathOperator{\td}{td}
\DeclareMathOperator{\mnw}{w}
\DeclareMathOperator{\mtd}{mtd}

\DeclareMathOperator{\bw}{bw}
\DeclareMathOperator{\bd}{bd}

\DeclareMathOperator{\cd}{cd}
\DeclareMathOperator{\dd}{dd}
\DeclareMathOperator{\cdd}{cdd}

\DeclareMathOperator{\csd}{c^{*}\hspace{-2pt}d}
\DeclareMathOperator{\dsd}{d^{*}\hspace{-2pt}d}
\DeclareMathOperator{\csdd}{c^{*}\hspace{-2pt}dd}
\DeclareMathOperator{\cdsd}{cd^{*}\hspace{-2pt}d}
\DeclareMathOperator{\csdsd}{c^{*}\hspace{-2pt}d^{*}\hspace{-3pt}d}

\newcommand{\csdepth}{c$^*$-depth\xspace}
\newcommand{\dsdepth}{d$^*$-depth\xspace}
\newcommand{\cdsdepth}{cd$^*$-depth\xspace}
\newcommand{\csddepth}{c$^*$d-depth\xspace}
\newcommand{\csdsdepth}{c$^*$d$^*$-depth\xspace}

\DeclareMathOperator{\rnk}{r}
\DeclareMathOperator{\clo}{cl}

\DeclareMathOperator{\funeq}{\sim_{\rm f}}
\DeclareMathOperator{\funle}{\le_{\rm f}}
\DeclareMathOperator{\funnleq}{\nleq_{\rm f}}

\DeclareMathOperator{\extcontr}{{\oslash}}
\DeclareMathOperator{\coextdel}{\mbox{\reflectbox{$\oslash$}}}

\newtheorem{problem}[theorem]{Problem}

\newcommand{\hy}{\hbox{-}\nobreak\hskip0pt}

\newcommand{\NPh}{$\mathsf{NP}$\hy{}hard\xspace}
\newcommand{\NPc}{$\mathsf{NP}$\hy{}complete\xspace}
\newcommand{\FPT}{$\mathsf{FPT}$\xspace}

\def\ve#1{\mathchoice{\mbox{\boldmath$\displaystyle\bf#1$}}
{\mbox{\boldmath$\textstyle\bf#1$}}
{\mbox{\boldmath$\scriptstyle\bf#1$}}
{\mbox{\boldmath$\scriptscriptstyle\bf#1$}}}

\newcommand\veb{{\ve b}}

\newcommand\veAm{{\ve A}}

\newcommand\vel{{\ve l}}

\newcommand\veu{{\ve u}}
\newcommand\vev{{\ve v}}

\newcommand\vex{{\ve x}}

\newcommand{\sep}{\;|\;}

\newcommand{\contract}{\mathbin{/}}

\newcommand{\ca}{\mathcal}
\newcommand{\seq}{\subseteq}

\usepackage{todonotes}
\usepackage{tikz-cd}

\usetikzlibrary{calc}

\begin{document}

\maketitle

\begin{abstract}
    Motivated by recently discovered connections between matroid
    depth measures and block-structured integer programming [ICALP 2020, 2022], we undertake
    a systematic study of recursive depth parameters for matrices and matroids,
    aiming to unify recently introduced and scattered concepts.
    We propose a general framework that naturally yields eight 
    different depth measures for matroids, 
    prove their fundamental properties and relationships,
    and relate them to two established notions in the field: matroid branch-depth
    and matroid tree-depth (a natural depth counterpart of matroid tree-width).
    In particular, we show that six of our eight measures are mutually functionally inequivalent,
    and among these, one is functionally equivalent to matroid branch-depth and another to matroid tree-depth.
    We also prove that all these depth measures, possibly
    except one, coincide on matroids and on matrices over any field,
    which, perhaps surprisingly, is not a trivial finding.
    Finally, we provide a comparison between the matroid parameters and classical depth measures of graphs.
\end{abstract}

\section{Introduction}

Structural parameters play a central role in combinatorial optimisation 
and algorithm design: they allow us to identify large classes of instances that, 
while intractable in general, admit efficient algorithms
when the underlying structure is sufficiently simple.

In graph theory, undoubtedly the most prominent structural parameter is tree-width,
whose importance is reflected in a wide range of structural and algorithmic consequences.
Among these is a fundamental result of Courcelle~\cite{courcelle1990monadic}, which shows that any property
expressible in monadic second-order logic (MSO) can be tested in linear time 
on graph classes of bounded tree-width. A closely related, but 
more restrictive parameter is the notion of tree-depth,
which plays a fundamental role in the sparsity theory of Nešetřil and Ossona de Mendez~\cite{sparsity-book}.

In the matroid setting, the analogue of graph tree-width and the most established notion is the parameter of branch-width. 
Numerous results known for graphs of bounded tree-width extend naturally to matroids of bounded branch-width.
Notably, Hliněný~\cite{Hli06} generalised Courcelle's Theorem to the matroid setting, 
proving that MSO properties of matroids represented over a finite field 
can be tested in polynomial time when branch-width is bounded.
As for depth parameters, two notions have been proposed as matroid analogues of tree-depth: {\em branch-depth}, introduced by
DeVos, Kwon, and Oum~\cite{DeVKO20}, and {\em contraction$^*$-depth}, defined independently by Kardoš,  Kr{\'{a}}l', Liebenau, and Mach~\cite{KarKLM17}.

Chan, Cooper, Kouteck\'y, Kr{\'{a}}l', and Pek{\'{a}}rkov{\'{a}}~\cite{ChaCKKP20,ChaCKKP22}
uncovered a surprising and powerful connection 
between matroid depth parameters and integer programming. 
In particular, they showed that matrices of bounded tree-depth -- an important tractable class of integer programs -- can be understood through 
the structure of the underlying vector matroids, 
and that suitable matroid depth measures capture exactly the minimum tree-depth attainable under row operations.

Building on this result, subsequent work has deepened both the structural and algorithmic understanding of these matroid parameters. 
Briański, Kouteck{\'{y}}, Kr{\'{a}}l', Pek{\'{a}}rkov{\'{a}}, and Schr{\"{o}}der~\cite{BriKKPS22,BriKKPS24} extended the results of Chan et al. 
to further matrix tree-depth variants and established additional connections between matroid depth measures and the structure of integer programs, 
in particular via bounds on the Graver bases of matrices, leading to broader fixed-parameter tractability results. 
Briański, Kráľ, and Lamaison~\cite{BriKL25} clarified the structural nature of contraction$^*$-depth 
by relating it to minimal contraction-depth extensions and showing that it admits characterisations analogous to those known for graph tree-depth.
Gajarský, Pekárková, and Pilipczuk~\cite{GajPP25} further investigated the obstructions for contraction$^*$-depth, 
contraction-depth, and deletion-depth of matroids representable over finite fields, and introduced a dual notion, deletion$^*$-depth.
Very recently, Bria\'nski, Hlin\v{e}n\'y, Kr{\'{a}}l', and Pek{\'{a}}rkov{\'{a}}~\cite{DBLP:journals/corr/abs-2402-16215prep} have shown
that matroid branch-depth is functionally (at most quadratically) equivalent to contraction$^*$-deletion$^*$-depth.

Motivated by these developments, we introduce and study a unifying framework for recursive depth measures of matroids,
which naturally yields a family of eight such different depth measures, out of which six are mutually functionally inequivalent.
The measures in our family exactly capture all aforementioned published depth measures except branch-depth,
and they include measures which are functionally equivalent to branch-depth, and also to the newly introduced matroid tree-depth,
which is a natural depth counterpart of the less known matroid tree-width of Hlin\v{e}n\'y and Whittle~\cite{matroid-tree-width}.
In addition to proving mutual comparisons of the known and new matroid depth measures, we, in particular,
prove that the recursively defined measures of our framework exactly coincide on matroids and on matrices over any field,
which---perhaps surprisingly---does not follow directly from their definitions.
We also give a comparison between these matroid parameters and classical depth measures of graphs.

\section{Basic terminology and notation}\label{sec:terminology}

In this preliminary section, we establish the notation and terminology used throughout our paper.
We denote the set of nonnegative integers by $\mathbb{N}$ and the set $\{1, \ldots, k\}$
of the first $k$ positive integers by $[k]$. Vectors are written in bold, such as $\ve v$, while their coordinates are in normal font.
To avoid a clash with the matroid deletion operator, we use the ordinary `minus' sign for set subtraction, i.e., we write $A-B$ for sets~$A,B$.

A \emph{parameter} (of a graph, a matroid, or a matrix) is a function that assigns a nonnegative integer to each argument.
A parameter $p$ is \emph{bounded} on a class of structures $\ca C$ if there exists $k \in \mathbb{N}$ such that $p(S) \leq k$ for every $S \in \ca C$.
We say that a parameter $p$ is \emph{functionally smaller} on $\ca C$ than a parameter $p'$, denoted $p \funle p'$, 
if there is a computable function $f$ such that $p(S) \le f(p'(S))$ for all~$S\in\ca C$.
If $p \funle p'$ and $p' \funle p$, we say that $p$ and $p'$ are \emph{functionally equivalent} and write $p \funeq p'$.

\subsection{Graphs}\label{sub:graphs}

All graphs considered in this paper are finite and undirected, and they may contain loops and parallel edges.
The \emph{radius} of a graph $G$ is the minimum $r \in \mathbb{N}$ such that there exists a vertex with distance at most $r$ from every vertex of $G$.

A \emph{rooted tree} $T$ is a connected acyclic graph with a specified vertex, called the \emph{root} of $T$.
The \emph{height} $\hght(T)$ of a rooted tree $T$ is the number of vertices on the longest root-to-leaf path of~$T$.
Given a vertex $v$ of a rooted tree $T$, we say that a vertex $u$ is an \emph{ancestor} of $v$ if $u$ lies on the unique path from $v$ to the root of $T$. Any vertex $w$ such that $v$ is an ancestor of $w$ is then called a \emph{descendant} of $v$. 
By the \emph{closure} $\cl(T)$ of a rooted tree $T$ we understand the graph obtained from $T$ by adding an edge from every vertex to each of its ancestors.
A \emph{rooted forest} $F$ is a graph whose each connected component is a rooted tree.
The \emph{height} of a rooted forest is the maximum height of its components.
The \emph{closure} $\cl(F)$ of a rooted forest $F$ is the graph obtained by taking the closure of every connected component of $F$.

The \emph{tree-depth} $\td(G)$ of a graph $G$ is the minimum height of a rooted forest $F$ such that the closure $\cl(F)$ contains $G$ as a subgraph.
Observe that in this definition, the tree-depth of a single-vertex graph is one, and the tree-depth of any star is 2.

\subsection{Matrices}
\label{sub:matrices}

If $\FF$ is a field, we denote the set of all matrices with $m$ rows and $n$ columns by $\FF^{\,m \times n}$.
Given a matrix $\veAm$, we denote the entry on its $i$-th row and $j$-th column by $A_{ij}$.
If $\veAm$ and $\veAm'$ are two matrices over $\FF$, we say that they are \emph{row-equivalent}
if one can be obtained from the other by performing elementary row operations.

In the context of integer programming, several variants of tree-depth of matrices are used.
Given a matrix $\veAm$, the \emph{primal graph} $G_P(\veAm)$ of $\veAm$ is the graph in which vertices one-to-one correspond to columns,
and two vertices are adjacent if $\veAm$ contains a row in which the corresponding entries are both nonzero.
Analogously, the \emph{dual graph} $G_D(\veAm)$ of a matrix $\veAm$ is the graph with vertices corresponding to rows
and two vertices being adjacent if there exists a column of $\veAm$ in which the corresponding entries are both nonzero.
Finally, the \emph{incidence graph} $G_I(\veAm)$ of a matrix $\veAm$ is a bipartite graph in which one part of vertices corresponds to rows,
the other to columns, and two vertices corresponding to a row $i$ and a column $j$ are adjacent if $A_{ij} \neq 0$.
The \emph{primal tree-depth} $\td_P(\veAm)$ of a matrix $\veAm$ is then the tree-depth of its primal graph $G_P(\veAm)$.
The \emph{dual} and the \emph{incidence tree-depth} are defined analogously as the tree-depth of the dual and incidence graphs of~$\veAm$.

\subsection{Matroids}\label{sub:matroids}

We follow the standard matroid terminology from the book of Oxley~\cite{Oxley}.
A \emph{matroid} is a pair $(E, \mathcal{I})$ where $E$ is a finite set and 
$\mathcal{I} \subseteq 2^E$ is a collection of subsets of $E$ satisfying the following three axioms:
\begin{enumerate}
    \item $\emptyset \in \mathcal{I}$.
    \item If $A \in \mathcal{I}$ and $B \subseteq A$, then $B \in \mathcal{I}$.
    \item If $A, B \in \mathcal{I}$ and $|A| > |B|$, then there exists an element $e \in (A-B)$ such that $B \cup \{e\} \in \mathcal{I}$.
\end{enumerate}

For a matroid $M = (E, \ca I)$,
the set $E$ is called the \emph{ground set} of $M$, and the sets in $\mathcal{I}$ are referred to as \emph{independent}.
We denote these sets by $E(M)$ and $\ca I(M)$, respectively.
A maximal independent set 
is called a \emph{basis}, while a minimal dependent set is called a \emph{circuit}.
We denote the set of all bases of $M$ by $\mathcal{B}(M)$, and the set of all circuits by $\mathcal{C}(M)$.

The \emph{dual matroid} $M^*$ of a matroid $M$ is one whose bases are set-complements of the bases of~$M$.
Formally, $\mathcal{B}(M^*)=\{B'\mid B'=E(M)-B \mbox{ for some }B\in\mathcal{B}(M)\}$.

Let $X \subseteq E$ be a set of matroid elements. The \emph{rank} $\rnk_M(X)$ of $X$ is defined as the cardinality of the largest independent subset of $X$.
The rank of $M$, denoted by $\rnk(M)$, is the rank of its ground set $E$.
An element $e \in E$ is called a \emph{loop} if $\rnk_M(\{e\}) = 0$ and a \emph{coloop} if $\rnk_{M^*}(\{e\})=0$ (that is, if $e$ is contained in every basis of $M$).
Two elements $e$ and $e'$ are \emph{parallel} if $\rnk_M(\{e\}) = \rnk_M(\{e'\}) = \rnk_M(\{e, e'\}) = 1$.

The rank function of a matroid satisfies the \emph{submodularity} condition; for any $X,Y\subseteq E(M)$, we have
$\rnk_M(X)+\rnk_M(Y)\geq\rnk_M(X\cup Y)+\rnk_M(X\cap Y)$.
If an equality holds there, then $(X,Y)$ is called a \emph{modular pair} in~$M$.
The \emph{closure} function of a matroid $M$ is defined on the subsets of $E(M)$ by $\cl_M(X)=\{e\mid \rnk_M(X\cup\{e\})=\rnk_M(X)\}$.

\paragraph*{Connectivity}

A \emph{component} of a matroid $M$ is an inclusion-wise maximal subset of elements of $M$ such that any two of its elements are contained in a common circuit.
A matroid is \emph{connected} if it has only one component.
In the case of matrices, a matrix $\ve A$ is connected if the matroid $M(\ve A)$ is connected, and the components of $\ve A$ are defined analogously.

The \emph{connectivity function} of a matroid $M$ is defined, for a set $X\subseteq E(M)$, as $\lambda_M(X)=\rnk_M(X)+\rnk_M(E(M) - X)-\rnk(M)$.
Note that in some literature, the connectivity function of matroids is defined with an additional~$+1$.
It is easy to observe that $\lambda_M(X)=0$ if $X$ is a component of~$M$.

\paragraph*{Matroid operations}

Let $M$ be a matroid and $e\in E(M)$.
The operation of \emph{deletion} of $e$, denoted by $M \setminus e$, yields the matroid 
$(E(M) - \{e\}, \{I \mid I \in \mathcal{I}\wedge I \subseteq E-\{e\}\})$. 
The \emph{contraction} of $e$ in $M$, denoted by $M\contract e$, is the operation dual to deletion, that is, $M\contract e=(M^*\setminus e)^*$.
In other words, $M\contract e$ yields the matroid $(E - \{e\}, \{I \mid (I \cup \{e\}) \in \mathcal{I}\wedge I \subseteq E-\{e\} \})$
if $e$ is not a loop, and $M\setminus e$ if $e$ is a loop.
These operations are naturally extended to sets $X$ in place of single $e$.

A matroid $M'$ is a \emph{restriction} of $M$ if $M'=M\setminus X$ for some~$X\subseteq E(M)$.
A matroid $M'$ is a \emph{minor} of $M$ if $M'$ is obtained by a sequence of element deletions and contractions from~$M$.
This is equivalent to stating that $M'=M\setminus X\contract Y$ for some disjoint~$X,Y\subseteq E(M)$.

A matroid $M'$ is an \emph{extension} (a \emph{coextension}) of $M$ by element $e$ if $M=M'\setminus e$ ($M=M'\contract e$) for some~$e\in E(M')$.
A coextension is thus a dual operation to an extension.
An extension $M'$ of $M$ by $e$ is called \emph{free} if all sets $Z\subseteq E(M)$ such that $e\in\cl_{M'}(Z)$ are of rank~$\rnk(M)$.

A pair $(X,Y)$ of disjoint sets $X,Y\subseteq E(M)$ is a \emph{bispan} if $E(M)=\cl_M(X)\cup\cl_M(Y)$.
A bispan $(X,Y)$ is a \emph{connected bispan} if $\rnk_M(X)+\rnk_M(Y)>\rnk(M)$.
An extension $M'$ of $M$ by $e$ is called \emph{relatively free in a bispan $(X,Y)$} of $M$ if $e\in\cl_{M'}(Z)$ for $Z\subseteq E(M)$,
if and only if $(X\cup Z,Y\cup Z)$ is a modular pair.
While a free extension trivially exists for every matroid, the existence of a relatively free extension in an arbitrary connected bispan
is a nontrivial feature, see \Cref{thm:GGWadde}.

\paragraph*{Matroid representations}

Two fundamental ways of representing matroids are the graphic and vector matroids. 
First, the \emph{cycle matroid} $M(G)$ of a graph $G$ is the matroid whose ground set is the set of edges $E(G)$,
and the independent sets are precisely all subsets of $E(G)$ that are acyclic in $G$. 

Second, given a matrix $\veAm$ over a field $\mathbb{F}$, the \emph{vector matroid} $M(\veAm)$ is the matroid whose ground set consists of the column vectors of $A$. 
A subset of elements is independent in this matroid if and only if the corresponding column vectors are linearly independent over $\mathbb{F}$.
Note that elementary row operations in $\ve A$ do not change the matroid~$M(\ve A)$.

If a matroid $M$ is isomorphic to the vector matroid of some matrix $\veAm$ over a field $\FF$, we say that $M$ is \emph{representable over} $\FF$ (or $\FF$\emph{-representable}),
and the matrix $\veAm$ is called a \emph{representation} (or an $\FF$\emph{-representation}).
If a matroid $M$ is associated with a particular $\FF$-representation $\ve A$ (precisely, with the row-equivalence class of $\ve A$),
then we speak about an $\FF$\emph{-represented matroid}.

In the case of a represented matroid $M=M(\ve A)$, when $\vev$ is a column of~$\ve A$ representing a non-coloop element~$e$,
then the operation of deleting $\vev$, denoted by $\ve A\setminus\vev$, simply removes the column $\vev$ from~$\ve A$,
and the resulting matrix represents~$M\setminus e$.
For the operation $\veAm \contract \vev$ of contracting a nonzero (non-loop) vector $\vev$ in~$\ve A$, we interpret this operation
as pivoting $\vev$ in $\ve A$ into a unit vector $\hat \vev$, and then removing the row with entry one in column $\hat \vev$, as well as the column $\hat \vev$ itself.
One can easily check that the resulting matrix represents~$M\contract e$.
Additionally, we use the following two operations, which differ from standard matrix contractions and deletions.
By $\veAm \extcontr \vev$, we denote the matrix obtained from $\veAm$ by adding a column vector $\vev$ to $\veAm$ and then contracting it.
Analogously, by $\veAm \coextdel \vev$, we denote the matrix obtained from $\veAm$ by adding a row vector $\vev$ to $\veAm$
(which actually represents the operation of coextending $\veAm$ with a new unit vector padded with $\vev$ in its new row and then deleting the column of this unit vector).

\section{Technical overview}\label{sec:toverview}

Our paper deals with a large number of mostly new technical concepts and their detailed properties.
To help readers navigate these concepts,
we begin with an extended overview which introduces the existing and the newly established depth measures and summarises our main results about them.
We refer to further \Cref{sec:depths} for exact definitions of the terms sketched here, 
and to the subsequent sections for the full formal statements and proofs.

\paragraph*{Matroid and matrix depth measures}
The tree-depth defined above is a prominent structural depth measure of graphs.
Although the name `tree-depth' was introduced only later within the sparsity theory of Ne\v{s}et\v{r}il and Ossona de Mendez~\cite{sparsity-book},
several equivalent concepts had been studied before.
In the context of this paper, it is interesting that the tree-depth of a graph~$G$ can be equivalently defined by the following recursion.
If $G$ consists of a single vertex, then $\td(G) = 1$. 
If $G$ is disconnected, consisting of components $G_1,\ldots,G_k$, then $\td(G)=\max\{\td(G_i)\mid i\in[k]\}$.
Finally, if $G$ is connected and has more than one vertex, then $\td(G)= 1 + \min\{\td(G-v)\mid v\in V(G)\}$.

Among matroid depth measures, the notions of branch-depth of \cite{DeVKO20} and contraction$^*$-depth of \cite{KarKLM17} (originally also
named `branch-depth' there) are both defined via decompositions, analogously to the first definition of tree-depth given in \Cref{sub:graphs}.

Namely, a \emph{branch-depth decomposition} (\Cref{def:branchdepth}) of a matroid $M$ is a rooted tree $T$ associated with a bijection 
from the elements $E(M)$ of $M$ to the leaves of $T$.
The width of an internal node $w$ of $T$ is the maximum connectivity of a set $X\subseteq E(M)$ where $X$ is the union of the elements
mapped to a subset of the components of~$T-w$.
The \emph{branch-depth} of $M$ is then the minimum $k$ such that there exists a branch-depth decomposition of $M$ of both radius and maximum width over its nodes at most~$k$.

\smallskip
On the other hand, one may choose to generalise the second, recursive approach to tree-depth.
Since matroids do not have vertices (very informally, they may be viewed as ``graphs consisting only of edges as the elements''),
such a definition has to remove matroid elements. This can be done in two ways -- by deletion or contraction.
To this end we say that a matroid $M'$ is \emph{c-transformed} (\emph{d-transformed}) from a matroid $M$ if
$M'$ is obtained by contracting an element of $M$, that is $M'=M\contract e$ (deleting an element of $M$, that is $M'=M\setminus e$, respectively)
for some~$e\in E(M)$, and $M'$ is \emph{cd-transformed} from $M$ if $M'$ is c-transformed or d-transformed from~$M$.
One can then define the following:

\begin{definition}\label{def:eight-prelim}
For {\rm$\gamma\in\{$`c',`d',`cd'$\}$}, the \emph{$\gamma$-depth} of a matroid $M$ is:
\begin{enumerate}[(i)]
\item $\gamma$-depth$(M)=1$ if $M$ has only one element, and
\item $\gamma$-depth$(M)=\max\{\gamma$-depth$(M')\mid M'$ is a component of~$M\}$ if $M$ is not connected, and
\item $\gamma$-depth$(M)= 1 + \min\{\gamma$-depth$(M')\mid M'$ is $\gamma$-transformed from~$M\}$ otherwise.
\end{enumerate}
\end{definition}

In~\cite{DeVKO20}, DeVos, Kwon, and Oum used \Cref{def:eight-prelim} to define the \emph{contraction-depth}, \emph{deletion-depth},
and \emph{contraction-deletion-depth} of matroids, which we abbreviate in this paper as the \emph{c-depth}, \emph{d-depth}, and \emph{cd-depth}.
Note that some other authors choose in \Cref{def:eight-prelim} to define point (i) as $\gamma$-depth$(M)=\rnk(M)$ when $M$ has one element.
This change, overall, decreases the resulting value by one, unless $M$ is of positive rank and consists of only loops and~coloops.
However natural it may seem, choosing $\gamma$-depth$(M)=\rnk(M)$ when $M$ has one element does not play well with matroid duality,
and hence we stick with exact \Cref{def:eight-prelim}.

When dealing with vector matroids represented by a matrix $\ve A$ (over an arbitrary field, with matroid elements being the column vectors), 
we define the $\gamma$-depth of $\ve A$ as the $\gamma$-depth of the matroid $M(\ve A)$.
In this view, a d-transformation simply means the removal of a column~of~$\ve A$, while a c-transformation is done by contracting (i.e., projecting along) a column vector of~$\ve A$.
This approach is taken, for instance, by the authors of the aforementioned~works~\cite{ChaCKKP22,BriKKPS24,GajPP25}.

Considering the depth measures of matrices -- that is, of actual vector representations of matroids -- brings one more natural possibility
for a $\gamma$-transformation in \Cref{def:eight-prelim} (iii):
instead of contracting a column of $\ve A$, we may contract an arbitrary $1$-dimensional subspace in the vector space of~$\ve A$,
abbreviated as being \emph{c$^*$-transformed}.
The dual operation to this, called \emph{d$^*$-transformation}, then simply corresponds to adding a new row to the matrix~$\ve A$.
This specific approach has been used, along the recursive lines of \Cref{def:eight-prelim}, to define 
the contraction$^*$-depth of matrices in \cite{ChaCKKP22} (see \Cref{def:csd-representable}), the contraction$^*$-deletion-depth of matrices 
in \cite{BriKKPS24} (see \Cref{def:csdd-representable}) and the deletion$^*$-depth of matrices in \cite{GajPP25} (see \Cref{def:dsd-representable}).
In the spirit of \Cref{def:eight-prelim},
we abbreviate these measures as the c$^*$-depth, c$^*$d-depth, and d$^*$-depth of matrices.

\paragraph*{Connection to Integer Programming}
The primary motivation for the study of matroid and matrix depth parameters came from a surprising connection
to the computational complexity of certain variants of integer programming (see \Cref{sub:IPconn} for more formal details). 
For a matrix $\veAm$, let $\td^*_P(\veAm), \td^*_D(\veAm)$, and $\td^*_I(\veAm)$ denote 
the minimum primal, dual, and incidence tree-depth, respectively, across all matrices that are row-equivalent to $\veAm$.
Then \cite{ChaCKKP22} and \cite{BriKKPS24} prove (\Cref{thm:min_td_equivalences}) that the minimum primal, dual, and incidence tree-depth, across all matrices that are row-equivalent to $\veAm$,
are equal (up to a possible difference of~$1$) in order to the d-depth, the c$^*$-depth and the \csddepth of the matrix $\veAm$ itself.

The occurrence of the starred depth measures c$^*$-depth and \csddepth in \Cref{thm:min_td_equivalences}, 
even outside of the IP-related application to matrices, naturally prompts the following question:
Can these measures be defined in a uniform way for abstract matroids, that is, without referring to a particular configuration of vectors in a vector space?
If so, then what are the structural and algorithmic properties and mutual relations of these depth measures?

Answering these and the associated basic theoretical questions constitutes the main new contribution of our paper.

\paragraph*{Towards unifying matroid-depth framework}

In response to the previous questions, we introduce a new unifying framework (see \Cref{sec:recursivemeas}) which, following the recursive scheme of \Cref{def:eight-prelim},
also incorporates the concepts of being `c$^*$-transformed' and `d$^*$-transformed' in purely matroidal terms.

Inspired by the concept of elementary operations of Geelen, Gerards, and Whittle~\cite{DBLP:journals/siamdm/GeelenGW06}, 
we define two operations that will take over the role of $c^*$- and $d^*$-transformations of matrices.

\begin{definition}\label{def:startransformation}
A matroid $M'$ is said to be a \emph{c$^*$-transformation} of a matroid $M$ if there exists a matroid $M^+$ and an element $f \in E(M^+)$ 
such that $f$ is not a loop nor coloop, $M = M^+ \setminus f$, and $M' = M^+ \contract f$.

Analogously, $M'$ is called a \emph{d$^*$-transformation} of $M$ if there is $M^+$ and $f \in E(M^+)$ 
such that $f$ is not a loop nor coloop, $M = M^+ \contract f$, and $M' = M^+ \setminus f$.
\end{definition}
Note that a d$^*$-transformation is the dual operation to a c$^*$-transformation.
\medskip

These two new operations then define the meaning of `$\gamma$-transformed' in \Cref{def:eight-prelim} for the symbols `c$^*$' and `d$^*$' within $\gamma$
(see \Cref{def:eight-parameters}).
Thus, for illustration, the c$^*$-depth of a matroid $M$ is defined as follows: (i) $1$ if $M$ has only one element,
(ii) the maximum of the c$^*$-depths of the components of $M$, if $M$ is disconnected, and (iii)
one plus the minimum of the c$^*$-depth over all matroids which are c$^*$-transformations of~$M$ otherwise.

A related question now is whether, for a matroid $M=M(\ve A)$ represented by a matrix~$\ve A$, the c$^*$-depth (or \csddepth) 
of the matrix $\ve A$ is equal to the c$^*$-depth of~$M$ itself.
We may ask further: if $\ve A_1$ and $\ve A_2$ are two matrices representing the same matroid
(but possibly over different fields), is it true that the c$^*$-depths of $\ve A_1$ and of $\ve A_2$ are equal?
Although the definitions of c$^*$-depth for matrices and matroids in the framework of \Cref{def:eight-prelim} look ``exactly the same'',
this is not a trivial question.
In the recursive process of computing the c$^*$-depth, one may choose to contract a $1$-dimensional subspace represented
by a vector which is over the field of $\ve A_1$, but it cannot be represented over the field of $\ve A_2$.
Or, as in \Cref{def:startransformation}, one may choose a c$^*$-transformation of a matroid $M$ by an element $f$ which is representable over no field~at~all.

These issues are not only technical, but they represent a real problem as illustrated, e.g., by further \Cref{ex:Fannostar} which shows a concrete non-representable case of c$^*$-transformation.
To clarify things, we prove the following (not-at-all trivial) statement:

\begin{theorem}[\Cref{cor:csdmatrixeq} and \Cref{cor:csxxmatrixeq}]\label{thm:matromatrix}
Let $\veAm$ be a matrix over a field $\FF$ and $M=M(\ve A)$ be the matroid represented by $\ve A$.
Then, for any {\rm$\gamma\in\{$`c$^*$', `d$^*$', `c$^*\!$d', `cd$^*$'$\,\}$},
the $\gamma$-depth of the matrix $\ve A$ equals the $\gamma$-depth of the matroid~$M$.
In particular, the matrix $\gamma$-depth indeed does not depend on a particular matrix representation of the same matroid.
\end{theorem}

\paragraph*{Comparing the depth parameters}

Based on extended \Cref{def:eight-prelim} (see \Cref{def:eight-parameters}), we obtain altogether eight different depth measures of matrices and/or abstract matroids, listed as the
c-depth, d-depth, c$^*$-depth, d$^*$-depth, cd-depth, c$^*$d-depth, cd$^*$-depth, and c$^*$d$^*$-depth.
There are also the two previously mentioned decomposition-based measures, the branch-depth by~\cite{DeVKO20} and the contraction$^*$-depth by~\cite{KarKLM17}.
However, the contraction$^*$-depth of a matroid $M$ is (by \Cref{thm:contractionstarx}) equal, up to a technical `minus one' difference, 
to the c$^*$-depth of~$M$, and hence we will not count the contraction$^*$-depth by~\cite{KarKLM17} as a different measure in our comparison.
Furthermore, there is also a natural depth counterpart of the less known matroid tree-width~\cite{matroid-tree-width}, which has not been
explicitly named in the literature so far, and which we introduce here (\Cref{sec:treede}) as the \emph{matroid tree-depth}.

\begin{figure}[t]
    \centering
    \begin{tikzpicture}[yscale={1}, xscale={2}, every node/.style={draw, very thin, fill=white, inner sep=5pt}]
        \node (csdsd) at (0,0) {\csdsdepth $\funeq~$branch-depth};
        \node (csdd) at (-2, -1.1) {\csddepth};
        \node (cdsd) at (2, -1.1) {\cdsdepth};
        \node (cdd) at (0, -2.2) {cd-depth};
        \node (cd) at (-2, -3.3) {c-depth $~\funeq~$\csdepth $\funeq~$m.~tree-depth};
        \node (dd) at (2, -3.3) {d-depth$~\funeq~$\dsdepth};
	\tikzstyle{every node}=[]
        \draw[<-] (csdsd) -- node[below] {\scriptsize f} (csdd) ; \draw[<-] (csdd) -- node[below] {\scriptsize f} (cdd) ;
	\draw[<-] (cdd) -- node[above] {\scriptsize f} (cd);
        \draw[<-] (csdsd) -- node[below] {\scriptsize f} (cdsd) ; \draw[<-] (cdsd) -- node[below] {\scriptsize f} (cdd) ;
	\draw[<-] (cdd) -- node[above] {\scriptsize f} (dd);
        \draw (csdsd) -- (cdsd) -- (cdd) -- (dd);
	\draw[dotted,<->] (-0.3,-3.3) -- (1,-3.3) node[midway, below=0pt] {\mbox{\scriptsize duality}};
	\draw[dotted,<->] (-0.8,-1.2) -- (0.8,-1.2) node[midway, below=0pt] {\mbox{\scriptsize duality}};
    \end{tikzpicture}
    \caption{A comparison of the considered matroid depth parameters in \Cref{thm:depthcompar}. 
	See \Cref{sec:terminology} for the definitions of functional comparison $p \funeq p'$ (``equal'') and $p \funle p'$ (``at most'') between measures.
	An arrow connection $p\to_{\rm f} q$ in the picture means functional ``strictly greater'', i.e., $q \funle p$ but $p \funnleq q$. 
	Moreover, c-depth and d-depth are dual to each other, but they are functionally incomparable both directions, as are \csddepth and \cdsdepth.}
    \label{fig:hasse-csdsd}
\end{figure}

We note that all the aforementioned parameters are well-defined, and the ``starred'' variants of the parameters are always less or equal to the corresponding ``non-starred'' variant. We defer the arguments to \Cref{rem:starstronger,rem:well-defined}.

We prove that, among these depth measures, there are six functionally inequivalent classes with mutual relations as summarised here:

\begin{theorem}[\Cref{cor:xd_circuits}, \Cref{lem:starredineq}, \Cref{thm:bd-csdsd}, \Cref{pro:allstrict}, \Cref{thm:mtreedepth-rel}]
\label{thm:depthcompar}%
Among the ten parameters, c-depth, d-depth, c$^*$-depth, d$^*$-depth, cd-depth, c$^*$d-depth, cd$^*$-depth, c$^*$d$^*$-depth,
branch-depth and matroid tree-depth, the functional equivalence and inequality relations depicted in \Cref{fig:hasse-csdsd} hold on the class of all matroids.
\end{theorem}

For instance, the functional equivalence of the measures c-depth and c$^*$-depth can be easily derived from existing results.
DeVos, Kwon, and Oum~\cite{DeVKO20} proved that the c-depth of a matroid $M$ is functionally related to the size of the longest circuit in $M$ (\Cref{thm:cd_circuits}),
while Kardoš et al.~\cite{KarKLM17} proved that the c$^*$-depth of a matroid $M$ is again functionally related to the size of the longest circuit in $M$ (\Cref{thm:csd_circuits}).
The difference between c$^*$-depth and c-depth can be up to exponential, as witnessed by the matroid formed by one long circuit.
The functional equivalence between d-depth and d$^*$-depth then follows by duality.

The proof of the functional equivalence between c$^*$-depth and matroid tree-depth (\Cref{thm:mtreedepth-rel}),
on the other hand, is a bit more involved and uses similar tools as the proof of \Cref{thm:matromatrix}.
The proof of the functional equivalence between c$^*$d$^*$-depth and branch-depth (\Cref{thm:bd-csdsd}) 
is one of the two core results of \cite{DBLP:journals/corr/abs-2402-16215prep}.

\paragraph*{Algorithmic aspects of depth measures}

It is known that deciding whether the branch-width is at most $k$ is \NPh~\cite{seymour1994call}.
However, for matroids represented over a finite field, a branch-decomposition of width at most $k$ can be efficiently computed in cubic time for fixed $k$
by~\cite{HliO08}. The latter result has been recently extended to the setting of abstract connectivity functions by Korhonen and Oum~\cite{korhonen2026branchwidthconnectivityfunctionsfixedparameter}.

The situation is somewhat similar for the considered depth measures.
In \cite{BriKKPS24}, it is proven (\Cref{thm:depthNPhard} and its immediate corollaries) that the problem -- whether the $\gamma$-depth of a matroid $M$ is at most $k$, is \NPh
for each $\gamma\in\{$`c',`c$^*$',`d',`d$^*$',`cd',`c$^*$d',`cd$^*$'$\}$.
On the other hand, \cite{ChaCKKP22} shows that, for matroids $M$ represented over a finite field, the same problem -- whether the $\gamma$-depth of $M$ is at most $k$,
is in \FPT\ for fixed~$k$ and $\gamma\in\{$`c$^*$',`d$^*$'$\}$.
The same is true for $\gamma\in\{$`c',`d'$\}$ by \cite{BriKKPS24} (\Cref{algo:fpt_dd}), and consequently also for $\gamma=\,$`cd' via a routine modification of \cite{BriKKPS24} (\Cref{cor:cdd_fpt}).

To this picture we add (\Cref{cor:csdd_fpt}) that the problem -- whether the $\gamma$-depth of a matroid $M$ is at most $k$, is
in \FPT\ for fixed~$k$ and $\gamma\in\{$`c$^*$d',`cd$^*$'$\}$ and matroids $M$ represented over a finite field,
which we can prove by a nontrivial adaptation of \cite{BriKKPS24}.

Interestingly, the complexity status of computing the branch-depth of matroids is still open~\cite{DeVKO20} (see \Cref{conj:bdhard}).

\paragraph*{Connection to depth measures of graphs}

Recall that every graph $G$ has a naturally associated matroid, called the \emph{cycle matroid} $M=M(G)$, where $E(M)=E(G)$ and
the independent sets of $M$ are the acyclic subsets of $E(G)$.
The traditional tree-depth of (all) graphs cannot be related to depth measures of matroids, as the tree-depth of trees
is unbounded, while all trees have a trivial independent cycle matroid.

On the other hand, it is well known (\Cref{thm:cycles-treedepth}) that the tree-depth of a $2$-connected graph $G$ is functionally
related to the length of the longest cycle in~$G$, and so the tree-depth of a $2$-connected graph $G$ is functionally
related to the c-depth and to the matroid tree-depth of the cycle matroid $M(G)$ (\Cref{cor:td-cd} via \Cref{thm:cd_circuits}).
Furthermore, \cite{DeVKO20} proved (\Cref{prop:3-connected}) that if one restricts to matroids that are the cycle matroids of $3$-connected
graphs, then the diagram in \Cref{fig:hasse-csdsd} collapses from c-depth to \csdsdepth, i.e., the c-depth of such graphic matroids is functionally
equivalent to their \csdsdepth.
An example of the cycle matroid of the graph $K_{3,n}$ shows (\Cref{pro:K3nexample}) that the analogous claim is false for the d-depth of such matroids.

There is also another, quite recent, notion of \emph{$2$-tree-depth} \cite{HuynhJMSW22}, which is defined
analogously to the recursive definition of ordinary tree-depth, but which considers the maximum of $2$-tree-depth of the blocks
for a non-$2$-connected graph (instead of the maximum over components of a disconnected graph).
While the $2$-tree-depth of any graph $G$ is upper-bounded by twice the cd-depth of the cycle matroid of~$G$ (\Cref{prop:td2-cdd}),
there is no functional upper bound on even the c$^*$d$^*$-depth of the cycle matroid of~$G$ in terms of the $2$-tree-depth of~$G$ (\Cref{prop:td2-no-csdsd}).
Further detailed discussion can be found in \Cref{sub:graphdepth}.

\paragraph*{Closure properties}

In \cite{BriKL25} (see \Cref{thm:contrclosure}), 
Bria\'nski, Kr{\'{a}}l', and Lamaison prove that the contraction$^*$-depth of a matroid $M$ is equal to the restriction closure of its c-depth,
that is, to the minimum c-depth of a matroid $M'$ that contains $M$ as a restriction
(modulo a technical difference by $1$ caused by the fact that \cite{BriKL25} treats the c-depth of a loop as $0$ and not~$1$ as we do).

We prove the analogous statement for the c$^*$-depth without using the results of \cite{BriKL25}:

\begin{theorem}[\Cref{thm:csclosure}]\label{thm:csclosurexx}
Let $M$ be a matroid. The minimum c-depth of a matroid $M'$ that contains $M$ as a restriction is equal to the c$^*$-depth of~$M$.
\end{theorem}

On the other hand, a small modification of our proof approach to \Cref{thm:csclosurexx} provides an alternative shorter proof
of the aforementioned closure result of \cite{BriKL25} (see \Cref{sub:contrstartproof}).

\section{Matroid depth and width parameters}\label{sec:depths}

After the preliminary technical overview of \Cref{sec:toverview}, we continue with the detailed definitions and statements, and the proofs.
In this section, we give a thorough overview of the existing and the newly introduced matroid depth parameters, starting with the classical two.

\subsection{Branch-width and branch-depth}\label{sub:branchd}

Branch-width is a classical structural width measure of matroids.
A \emph{subcubic} tree is a tree whose all vertices have degree at most~$3$.

\begin{definition}\label{def:branchwidth}
A \emph{branch-decomposition} of a matroid $M$ is a pair $(T, \sigma)$ where $T$ is a subcubic tree and $\sigma$ is a bijection from $E(M)$ to the set of leaves of $T$.
Let $e$ be an edge of $T$, let $T_1, T_2$ denote the connected components of $T \setminus e$, and let $L_1, L_2$ denote the sets of leaves of $T_1$ and $T_2$.
The \emph{width} of an edge $e$ in $T$, denoted by $w_T(e)$, is defined as $w_T(e) := \lambda(\sigma^{-1}(L_1)) = \lambda(\sigma^{-1}(L_2))$. 

The \emph{width} of a branch-decomposition $(T, \sigma)$ is the maximum width of any edge of the decomposition.
Finally, the \emph{branch-width} of a matroid $M$, denoted by $\bw(M)$, is the minimum width over all branch-decompositions of $M$.
\end{definition}

Note that branch-width is \emph{self-dual}, meaning that for any matroid $M$, the branch-width of $M$ is the same as the branch-width of the dual~$M^*$.

\medskip
In~\cite{DeVKO20}, DeVos, Kwon, and Oum defined the notion of \emph{branch-depth}
as the analogue of tree-depth in the matroid setting.
Although the authors defined the parameter in the general context of connectivity functions, we recall here only its definition for matroids.

For a matroid $M$ and a bijection $\sigma$ from $E(M)$ to the set of leaves of a tree $T$, we define the {width of an inner node $v\in V(T)$}
as follows. Let $L_1,\ldots,L_k$ denote the sets of leaves (of~$T$) belonging to the $k\geq2$ components of $T-v$, and let $\ca L_v=\{\sigma^{-1}(L_1),\ldots,\sigma^{-1}(L_k)\}$.
Then the \emph{bd-width of $v$} is $\nu_M(\ca L_v):=\max_{\ca L'\seq\ca L_v} \lambda_M\big(\bigcup \ca L'\big)$, that is, the maximum connectivity of a bipartition
coarsening the partition of $E(M)$ given by the subtrees of $v$ in~$T$.

\begin{definition}[\cite{DeVKO20}]\label{def:branchdepth}
A \emph{branch-depth decomposition} of a matroid $M$ is pair $(T, \sigma)$ where $T$ is a tree with at least one inner node 
and $\sigma$ is a bijection from $E(M)$ to the leaves of $T$.
We say that a branch-depth decomposition $(T, \sigma)$ is a $(t,d)$\emph{-decomposition} if the bd-width of every inner node of $T$ is at most $t$
and the radius of $T$ is at most $d$.
The \emph{branch-depth} of a matroid~$M$, denoted by $\bd(M)$, is the minimum $k \in \mathbb{N}$ such that there exists a $(k, k)$-decomposition of $M$.
In case $|E(M)| \leq 1$, no branch-depth decomposition is defined and~$\bd(M)=0$.
\end{definition}

\subsection{Tree-width and tree-depth}\label{sec:treede}

Classical tree-width of graphs has a direct extension in matroid tree-width, which was defined by Hlin\v{e}n\'y and Whittle~\cite{matroid-tree-width,DBLP:journals/ejc/HlinenyW09}.
Due to crucial differences between graphs and matroids (such as the absence of vertices in the latter), a definition of matroid tree-width has to be
formulated very differently from the classical graph definition, albeit these two definitions give exactly the same values
on graphs and their cycle matroids \cite[Theorem~3.2]{matroid-tree-width}.

We will use the following definition, for a matroid $M$ and a collection of $k\geq1$ pairwise disjoint sets $X_1,\ldots,X_k\subseteq E(M)$:
let $\omega_M(X_1,\ldots,X_k)=\sum_{i=1}^k \rnk_M(E(M)-X_i)\>-(k-1)\rnk(M)$.
Note that the function $\omega_M$, similarly to the function $\nu_M$ from \Cref{sub:branchd}, generalises the traditional
matroid connectivity function $\lambda_M$, albeit in a different way. We also have:

\begin{lemma}\label{lem:omegaMzero}
Let $M$ be a matroid and $X_1,\ldots,X_k\subseteq E(M)$ a collection of $k\geq1$ pairwise disjoint sets.
Then $\omega_M(X_1,\ldots,X_k)\geq0$.
\end{lemma}
\begin{proof}
We proceed by straightforward induction on $k\geq1$. The statement is trivial for $k=1$ since $\omega_M(X_1)=\rnk_M(E(M)-X_1)\geq0$.
For $k>1$ we, by submodularity, obtain
{\small
\begin{align*}
\omega_M(X_1,\ldots,X_k)&= \sum_{i=1}^k \rnk_M(E(M)-X_i)\>-(k-1)\rnk(M)
\\	&\geq \rnk_M\big((E(M)-X_1)\cap(E(M)-X_2)\big)+\sum_{i=3}^k \rnk_M(E(M)-X_i)\>-(k-2)\rnk(M)
\\	&=\omega_M(X_1\cup X_2,X_3,\ldots,X_k) \geq0
.\qedhere
\end{align*}}
\end{proof}

\begin{definition}[\cite{matroid-tree-width}]\label{def:mtreewidth}
A \emph{tree-decomposition} of a matroid $M$ is a pair $(T, \tau)$ such that $T$ is a tree and $\tau\colon E(M) \rightarrow V(T)$ is an arbitrary mapping.
For a node $u \in V(T)$ of degree $d$, let $T_1,\ldots,T_d$ be the connected components of $T-u$, and let $F_i = \tau^{-1}(V(T_i))$ for $i \in [d]$.
The \emph{td-width} of the node $u$ is defined as $\mnw_{M,T}(u):=\omega_M(F_1,\ldots,F_d)$

The \emph{width} of a tree-decomposition $(T, \tau)$ is the maximum td-width over all nodes of $T$, and the \emph{matroid tree-width} of $M$ 
is the smallest width over all tree-decompositions~of~$M$.
\end{definition}

Matroid tree-width is within a constant factor of matroid branch-width \cite[Theorem~4.2]{matroid-tree-width}.
The easy potential of \Cref{def:mtreewidth} for defining a notion analogous to graph tree-depth has been overlooked by researchers so far.
The following definition thus naturally brings a new parameter:

\begin{definition}\label{def:mtreedepth}
The \emph{matroid tree-depth} of a matroid $M$, denoted by $\mtd(M)$, is the minimum $k \in \mathbb{N}$ for which there exists 
a {tree-decomposition} $(T, \tau)$ of $M$, such that the width of $(T, \tau)$ is at most $k$ and the radius of $T$ is also at most~$k$.
\end{definition}

Matroid tree-depth of a graphic matroid $M(G)$ is functionally equivalent to tree-depth of the graph $G$ provided that $G$ is $2$-connected (cf. \Cref{cor:func-equiv-td-mtd}).

Matroid tree-width is functionally preserved under duality; precisely, \cite[Theorem~4.2]{matroid-tree-width}
implies $\mtw(M)$ and $\mtw(M^*)$ are within a multiplicative constant of each other for all matroids~$M$.
Unfortunately, one direction of the proof of \cite[Theorem~4.2]{matroid-tree-width} heavily changes the decomposition tree,
and so this does not imply anything useful about the relation between matroid tree-depth of a matroid and of its dual.

In fact, it is easy to construct a sequence of matroids~$M$ (\Cref{lem:incom-cddd}; simply, the sequence $U_{n-1, n}$ for all $n$) 
such that $\mtd(M^*)$ is bounded by a constant, while $\mtd(M)$ is unbounded.
This observation shows that there is a potential for a future definition of a generalised version of matroid tree-depth, which would be closed under duality.

\subsection{The family of ``recursive'' depth measures}
\label{sec:recursivemeas}

Over the past decade, several new depth measures for both represented and general matroids have been introduced.
All of them are based on a common principle of recursively decomposing a matroid into components, 
inspired by the notion of tree-depth in graphs.
We begin the overview of these notions with a common generic definition (extending \Cref{def:eight-prelim}), which we subsequently relate to the existing measures.

Recall the notions of c- and d-transformations used in \Cref{def:eight-prelim}, and of c$^*$- and d$^*$-transformations from \Cref{def:startransformation}.

\begin{definition}\label{def:eight-parameters}
Let $\gamma \in \{\epsilon, $`c',`c$^*\!$'$\,\}$ and $\delta \in \{\epsilon, $`d',`d$^*\!$'$\,\}$ be symbols, where $\epsilon$ stands for the empty word.
Assuming $(\gamma,\delta)\not=(\epsilon,\epsilon)$, the \emph{$\gamma\delta$-depth} of a matroid $M$ is defined as follows.
\begin{enumerate}[(I)]
\item If $M$ has at most one element, then the $\gamma\delta$-depth of $M$ is $1$.\label{it:one-element}
\item If $M$ is not connected, then the $\gamma\delta$-depth of $M$ is the maximum $\gamma\delta$-depth of a component of~$M$.
Symbolically,~ \emph{$\gamma\delta$-depth$(M)=\max\{\gamma\delta$-depth$(M'): M'$ a component of $M\}$}.
\label{it:bycomponents}
\item If $M$ is connected, then the $\gamma\delta$-depth of $M$ is one plus the minimum $\gamma\delta$-depth
  of a matroid $M'$ which is $\gamma\delta$-transformed from $M$, where by `$\gamma\delta$-transformed from $M$' we mean:
\label{it:oneredstep}
\begin{enumerate}[a)]
\item $\gamma = c$ and $M' = M \contract e$ for some $e \in E(M)$, or
\item $\delta = d$ and $M' = M \setminus e$ for some $e \in E(M)$, or
\item $\gamma = c^*$ and $M'$ is obtained by a c$^*$-transformation of~$M$, or
\item $\delta = d^*$ and $M'$ is obtained by a d$^*$-transformation of~$M$.
\end{enumerate}
\end{enumerate}
Hence, symbolically,~ \emph{$\gamma\delta$-depth$(M)=1+\min\{\gamma\delta$-depth$(M'): M'$ is ${\gamma\delta}$-transformed~from~$M\}$}.

We abbreviate `$\gamma\delta$-depth' by removing the empty-word symbol $\epsilon$, that is, we speak about the following eight
measures coming from the different possible combinations of our symbols: \emph{c-depth, d-depth, c$^*$-depth, d$^*$-depth, cd-depth, c$^*$d-depth, cd$^*$-depth}
and \emph{c$^*$d$^*$-depth}.
The corresponding $\gamma\delta$-depth of a matroid $M$ is then shortly denoted by $\cd(M)$, $\dd(M)$, $\csd(M)$, $\dsd(M)$, $\cdd(M)$, 
$\csdd(M)$, $\cdsd(M)$, and $\csdsd(M)$, respectively.
\end{definition}

\begin{remark}\label{rem:starstronger}
Note that the starred variants of operators in \Cref{def:eight-parameters},~\Cref{it:oneredstep} (i.e., c$^*$ and d$^*$) are always stronger than 
the non-starred ones in the following sense; instead of contracting an element $e\in E(M)$ in \Cref{it:oneredstep}(a), we can introduce an
element $f$ in parallel to $e$ into $M$ and contract $f$ instead in \Cref{it:oneredstep}(c).
Element $e$ then becomes a loop, which will not change anything by \Cref{def:eight-parameters},~\Cref{it:bycomponents} in the next step 
if $M$ has more than one element -- see also the proof of \Cref{lem:starredineq}.
An analogous dual argument holds for \Cref{it:oneredstep}(b) and \Cref{it:oneredstep}(d).
Hence, in particular, we have $\csd(M)\leq\cd(M)$, $\dsd(M)\leq\dd(M)$, $\csdd(M) \le \cdd(M)$, $\cdsd(M) \le \cdd(M)$, and $\csdsd(M)\leq\cdd(M)$ for all matroids~$M$.
\end{remark}
\begin{remark}\label{rem:well-defined}
Note that all variants of depth measures covered by \Cref{def:eight-parameters} are well-defined;
in the worst-case scenario, we may simply contract or remove all elements of the matroid one by one.
For the starred variants of the measures, \Cref{rem:starstronger} additionally~applies.
\end{remark}

We proceed with a brief overview of the matroid depth measures in existing literature that are covered by \Cref{def:eight-parameters}.

\paragraph*{D-depth, c-depth, and cd-depth} 

Together with branch-depth (\Cref{sub:branchd}), DeVos, Kwon, and Oum~\cite{DeVKO20} introduced three additional depth measures: \emph{deletion-depth}, 
\emph{contraction-depth}, and \emph{contraction–deletion-depth}. 
These parameters coincide with the d-depth, c-depth, and cd-depth from \Cref{def:eight-parameters}. 
To keep the terminology uniform throughout the paper, we henceforth use the names d-depth, c-depth, and cd-depth. 
We also note that cd-depth already appeared in~\cite[Section~3]{DING199521} under the name \emph{type}.

\paragraph*{\boldmath Contraction$^*$-depth}

Another measure in the batch was defined in the work of Kardo\v{s}, Kr{\'{a}}l', Liebenau, and Mach~\cite{KarKLM17},
originally under the name `{branch-depth}' and later renamed to \mbox{`contraction$^*$-depth'}.

\begin{definition}[\cite{KarKLM17}]
\label{def:csd-decomposition}\label{def:csd-general}
    A \emph{contraction$^*$-depth decomposition} of a matroid $M$ is a pair $(T, f)$ where $T$ is a rooted tree and $f$ a mapping of
    the elements of $M$ into the leaves of $T$ such that
    \begin{itemize}
        \item $|E(T)| = \rnk(M)$, and
        \item for every $X \subseteq E(M)$ we have $\rnk_M(X) \leq |E(T_X)|$, where $T_X\subseteq T$ denotes the subtree formed
	by the union of all paths from the root to the leaves $f(x)$ where~$x\in X$.
    \end{itemize}
    The \emph{contraction$^*$-depth} of a matroid $M$ is the minimum value of $\,\hght(T)-1$ (minimum height minus one) of a contraction$^*$-depth decomposition of $M$.
\end{definition}

Although \Cref{def:csd-decomposition} has a very different form from \Cref{def:eight-parameters},
the defined parameter is nevertheless closely related to our family of measures.
In fact, Bria\'nski, Kouteck{\'{y}}, Kr{\'{a}}l', Pek{\'{a}}rkov{\'{a}}, and Schr{\"{o}}der~\cite{BriKKPS24} later
gave a definition of another parameter named contraction$^*$-depth for vector matroids represented by a matrix $\ve A \in \FF^{m \times n}$
in the style of the c$^*$-depth of \Cref{def:eight-parameters}.
There is only one minor difference (cf.~\Cref{rem:diff1}) between their definition and ours: the depth of a single-column matrix $\ve A$ in \cite{BriKKPS24} is defined
as $\rnk(\ve A)\in\{0,1\}$, instead of a fixed value $1$ as in \Cref{def:eight-parameters}.
In order to keep this paper uniform, we modify the definition of \cite{BriKKPS24} in this respect and use the name `c$^*$-depth'
in the forthcoming \Cref{def:csd-representable}.

For shorthand, we say that a depth measure $\mu$ of matrices $\veAm\in \FF^{m \times n}$ is \emph{component-shattered} if both of the following two conditions hold:
\begin{itemize}
    \item (i) $\mu(\ve A)=1$ if $n=1$ (i.e., the matroid $M(\ve A)$ has one element), and
    \item (ii) $\mu(\veAm)=\max\{\mu(\ve A'): \ve A'$ is a component of~$\ve A\}$.
\end{itemize}

\begin{definition}[adapted from \cite{BriKKPS24}]
    \label{def:csd-representable}
    Let $\FF$ be a field and $\veAm \in \FF^{m \times n}$ a matrix over $\FF$ for positive integers~$m,n$.
    The \emph{c$^*$-depth} of~$\veAm$, denoted by $\csd(\veAm)$, is a component-shattered depth measure 
    additionally satisfying the following condition:
    \begin{itemize}
        \item If $M(\ve A)$ is connected, then $\csd(\veAm) = 1 + \min\{ \csd(\veAm \extcontr \vev): \vev \in \FF^m$ is any column vector$\}$.
    \end{itemize}
Recall that $\veAm \extcontr \vev$ denotes the matrix obtained by adding a column $\ve v$ to $\ve A$ and contracting it.
\end{definition}

\begin{remark}\label{rem:diff1}
The c$^*$-depth, in the variant used in \cite{BriKKPS24}, of a matrix~$\veAm$ over $\FF$ equals
\begin{itemize}
\item $1=\csd(\veAm)$ if every element of $M(\ve A)$ is a loop or a coloop, and the rank of $\ve A$ is not null,
\item and $(\csd(\veAm)-1)$ otherwise.
\end{itemize}
\end{remark}

Although \Cref{def:csd-representable} may appear identical to \Cref{def:eight-parameters} of c$^*$-depth
of the matroid $M(\ve A)$, this is not completely easy, as already discussed in \Cref{sec:toverview}.
In a nutshell, \Cref{def:eight-parameters} allows one to contract
an arbitrary matroid element added to $M(\ve A)$, whereas \Cref{def:csd-representable} is restricted to adding and contracting only $\FF$-represented elements.
At first glance, it is not even clear whether $\csd(\ve A_1)=\csd(\ve A_2)$ holds for all representations of the same matroid $M(\ve A_1)=M(\ve A_2)$.
This is indeed a real problem:

\begin{example}\label{ex:Fannostar}
Consider the following example with the Fano matroid~$F_7$ (i.e., the projective plane over the binary field), which is illustrated in \Cref{fig:fanoex}.
The matroid $M_1=F_7\contract f$ (where $f\in E(F_7)$) consists of three parallel pairs in a line, and is representable over, say, the reals $\mathbb{R}$.
The matroid $M_2=F_7\setminus f$ is representable over the reals $\mathbb{R}$, too.
$M_1$ is a c$^*$-transformation of~$M_2$ by definition.

However, if there was an $\mathbb{R}$-representation $\ve A$ of $M_2$ such that $\ve A$ can be c$^*$-transformed to a representation of $M_1$,
then the three lines of $M_2$ which formerly contained $f$ in $F_7$ would need to have a point in common in the representation $\ve A$,
in order to yield a representation of $M_1$ after any contraction.
Hence, we would get an $\mathbb{R}$-representation of the Fano $F_7$, which is impossible.
In other words, we have an example of a c$^*$-transformation between two $\mathbb{R}$-representable matroids which itself is not $\mathbb{R}$-representable.

Similar examples can be built over other fields.
\end{example}

\begin{figure}[t]
    \centering
\begin{tikzpicture}[xscale=1.5,
    every node/.style={circle, fill=black, inner sep=1.2pt, minimum size=1.6pt},
]
    \node (A) at (90:2) {};
    \node (B) at (330:2) {};
    \node[label=above left:$f$] (C) at (210:2) {};
    \node (D) at (30:1) {};
    \node (E) at (150:1) {};
    \node (F) at (270:1) {};
    \node (G) at (0,0) {};
    \draw (A) -- (B) -- (C) -- (A); 
    \draw (A) -- (F); 
    \draw (B) -- (E); 
    \draw (C) -- (D); 
    \draw (G) circle (9.9mm);
\end{tikzpicture}
    \caption{The Fano matroid $F_7$ -- the projective plane over the binary field (its $7$ lines are the six line segments and the central cycle), 
	which is not representable over fields whose characteristic is different from~$2$. }
    \label{fig:fanoex}
\end{figure}

To our knowledge, this problem has not been explicitly addressed in the literature yet,
and we prove (\Cref{thm:matromatrix}, in \Cref{cor:csdmatrixeq}) that the c$^*$-depth $\csd(\veAm)$ of any matrix $\ve A$ indeed 
equals the \mbox{c$^*$-depth} of the matroid $M(\ve A)$, to provide a complete picture.
In particular, the matrix c$^*$-depth $\csd(\ve A)$ does not depend on a particular choice of a matrix $\ve A$ representing the same matroid.

Likewise, the relation of \Cref{def:csd-decomposition} to \Cref{def:csd-representable} is not at all clear, and it is only
implicitly claimed in \cite{BriKKPS24}, without a dedicated proof.
However, this relation can be derived from the results of Bria\'nski, Kr{\'{a}}l', and Lamaison~\cite{BriKL25}, as we detail in \Cref{sub:understand-csd}:
\begin{theorem}[using \cite{BriKL25}]\label{thm:contractionstarx}
Let $M$ be a matroid and $k$ denote the contraction$^*$-depth of~$M$ (according to \Cref{def:csd-decomposition}).
Then $k=\csd(M)-1$, unless $M$ is of positive rank and consists of only loops and coloops, in which case~$k=\csd(M)=1$.
\end{theorem}
\noindent We will thus further use the notion of c$^*$-depth and \Cref{thm:contractionstarx} when referring to existing results on contraction$^*$-depth.

We remark that further refined properties of contraction$^*$-depth of matroids are provided in \Cref{sec:closurede};
these also imply an alternative shorter proof for the closure result of \cite{BriKL25}.

\paragraph*{\boldmath C$^*$d-depth and d$^*$-depth of matrices}

Bria\'nski, Kouteck{\'{y}}, Kr{\'{a}}l', Pek{\'{a}}rkov{\'{a}}, and Schr{\"{o}}der~\cite{BriKKPS24} moreover defined
a new depth measure called `contraction$^*$-deletion-depth'
of represented matroids, in a way completely analogous to \Cref{def:csd-representable}, which we again abbreviate as follows.

\begin{definition}[\cite{BriKKPS24}]
    \label{def:csdd-representable}
    Let $\FF$ be a field and $\veAm \in \FF^{m \times n}$ a matrix over $\FF$.
    The \emph{c$^*$d-depth} of~$\veAm$, denoted by $\csdd(\veAm)$, is a component-shattered depth measure 
    additionally satisfying the following condition:
    \begin{itemize}
    \item If $M(\ve A)$ is connected, then $\csdd(\veAm) = 1 + \min(P\cup Q)$ where 
	$P=\{ \csdd(\veAm \extcontr \vev): \vev \in \FF^m$ is any column vector$\}$ and
	$Q=\{ \csdd(\veAm \setminus \ve w): \ve w$ is a column of $\ve A\}$.
    \end{itemize}
\end{definition}

Again, in comparison with \Cref{def:eight-parameters}, we prove in \Cref{cor:csxxmatrixeq}
that the c$^*$d-depth $\csdd(\veAm)$ of any matrix $\ve A$ equals the c$^*$d-depth of the underlying matroid $M(\ve A)$.

Finally, in~\cite{GajPP25}, Gajarsk\'y, Pek{\'{a}}rkov{\'{a}}, and Pilipczuk introduced a new parameter of represented matroids 
dual to c$^*$-depth (\Cref{def:csd-representable}), named in their work `{deletion$^*$-depth}', and here again abbreviated as d$^*$-depth.

\begin{definition}[\cite{GajPP25}]
    \label{def:dsd-representable}  
    Let $\FF$ be a field and $\veAm \in \FF^{m \times n}$ a matrix over $\FF$.
    The \emph{d\/$^*$-depth} of~$\veAm$, denoted by $\dsd(\veAm)$, is a component-shattered depth measure 
    additionally satisfying the following condition:
    \begin{itemize}
        \item If $M(\ve A)$ is connected, then $\dsd(\veAm) = 1 + \min\{ \dsd(\veAm \coextdel \vev): \vev \in \FF^n$ is any row vector$\}$.
    \end{itemize}
Recall that $\veAm \coextdel \vev$ denotes the matrix obtained by adding a row $\ve v$ to $\ve A$.
\end{definition}

Dually to the case of c$^*$-depth, 
the d$^*$-depth $\dsd(\veAm)$ of any matrix~$\ve A$ equals the d$^*$-depth of the matroid $M(\ve A)$.

\subsection{Connection to Integer Programming}\label{sub:IPconn}

In~\cite{ChaCKKP22}, Chan, Cooper, Kouteck\'y, Kr{\'{a}}l', and Pek{\'{a}}rkov{\'{a}} discovered a surprising connection between 
matroid depth parameters and integer programming, in particular, integer programs of bounded primal tree-depth.
By \emph{Integer programming (IP)}, we mean the following problem:
\begin{align*}
    \min \{f(\vex) \mid \veAm\vex = \vev,\hspace{0.5em} \vel \leq \vex \leq \veu,\hspace{0.5em} \vex \in \mathbb{Z}^n\}.
\end{align*}

Here, $\veAm$ is a $\mathbb{Z}^{m \times n}$ matrix (the \emph{constraint matrix}), $f: \mathbb{R}^n \rightarrow \mathbb{R}$ a separable convex function (the \emph{objective function}),
$\veb \in \ZZ^m$ the right-hand side vector, and $\vel, \veu \in (\ZZ \cup \{\pm \infty\})^n$ the lower and upper bounds.
Integer programming is known to be \NPh~\cite{Kar72} in general. However, when the constraint matrix has a certain restricted structure, the problem often becomes tractable.
One notable class of integer programs for which tractability has been established is the class of programs whose constraint matrices have bounded tree-depth~\cite{KoulO18}.

Unfortunately, tree-depth as a matrix parameter has the disadvantage that it is not invariant under row operations. 
As a result, two integer programs that are equivalent in terms of their solution sets may have different tree-depths, and 
an algorithm that performs efficiently on one formulation might be slow -- or even infeasible -- on the other.

To overcome this drawback, Chan et al.~\cite{ChaCKKP22} and later Briański et al.~\cite{BriKKPS24} studied the problem of matrix sparsification.
That is, given a matrix of possibly large tree-depth, to determine whether there exists a row-equivalent form of smaller tree-depth,
and whether it is possible to efficiently unveil this form. 
Matroid theory, and in particular matroid depth parameters introduced above, were the central tool in answering these questions.

To this end, we denote by $\td^*_P(\veAm)$ ($\td^*_D(\veAm)$, $\td^*_I(\veAm)$)
the minimum value of $\td_P(\veAm')$ ($\td_D(\veAm')$, $\td_I(\veAm')$, respectively) over all matrices $\veAm'$ which are row-equivalent to $\veAm$.

\begin{theorem}[\cite{ChaCKKP22} and \cite{BriKKPS24}]
\label{thm:min_td_equivalences}
For every matrix $\veAm$ it holds that
\begin{enumerate}[(a)]
\item the minimum primal tree-depth of a matrix row-equivalent to $\veAm$ is equal to the d-depth of $\veAm$,
    that is, $\td^*_P(\veAm)=\dd(\veAm)$,
\item the minimum dual tree-depth of a matrix row-equivalent to $\ve A$ is equal to the c$^*$-depth of $\veAm$ decreased by one,
    that is, $\td^*_D(\veAm)=\csd(\veAm) - 1$, unless $\ve A$ is of positive rank and $M(\ve A)$ consists only of loops and 
    coloops, in which case $\td^*_D(\veAm)=\csd(\veAm)$,
\item the minimum incidence tree-depth of a matrix row-equivalent to $\ve A$ is equal to the \csddepth of $\veAm$ increased by one, that is, $\td^*_I(\veAm)=\csdd(\veAm) + 1$.
\end{enumerate}
\end{theorem}

\section{Basic parameter properties}\label{sec:paramproperties}

This section establishes the basic structural properties 
of the studied matroid depth parameters, based partly on existing published results
and partly on new findings.

\subsection{Understanding \boldmath c$^*$-depth and contraction$^*$-depth}\label{sub:understand-csd}

In this section, we prove \Cref{thm:contractionstarx}, establishing the relationship between c$^*$-depth and contraction$^*$-depth 
of matroids.
We begin with three technical lemmas that reveal some fine details of the recursive definition of c$^*$-depth. 
We use these lemmas to relate c$^*$-depth to contraction$^*$-depth in this section, and also as tools to
establish further structural results in later sections.

\begin{lemma}\label{lem:extcd-to-csd}
If a matroid $M$ is a minor of a matroid~$M'$, then~$\csd(M)\leq\csd(M')$.
\end{lemma}

\begin{proof}
Let $X,Y\subseteq E(M')$ be such that $M=M'\contract X\setminus Y$ (hence~$E(M')=E(M)\cup X\cup Y$).
We will proceed by structural induction on \Cref{def:eight-parameters}.
If $|E(M')|=1$, then $\csd(M)\leq\csd(M')=1$ trivially holds.
If $M'$ is not connected, then every component $M_1$ of $M$ is a minor of a component $M_1'$ of $M'$.
By induction, $\csd(M_1)\leq\csd(M_1')$, and by \Cref{def:eight-parameters}\,(\ref{it:bycomponents}), $\csd(M)\leq\csd(M')$.

Assume that $M'$ is connected, and that by \Cref{def:eight-parameters}\,(\ref{it:oneredstep} a), $\csd(M')=1+\csd(M'')$
where $M''$ is a c$^*$-transformation of $M'$, i.e., $M'=M^+\setminus f$ and $M''=M^+\contract f$ for some matroid $M^+$.
We have $M=M^+\setminus f\contract X\setminus Y=(M^+\contract X\setminus Y)\setminus f$.
Let $M_1=(M^+\contract X\setminus Y)\contract f=M^+\contract f\contract X\setminus Y$.
Then $M_1$ is a c$^*$-transformation of $M$, and $M_1$ is a minor of $M''$.
By induction, $\csd(M_1)\leq\csd(M'')$, and by \Cref{def:eight-parameters}\,(\ref{it:oneredstep}),
$\csd(M)\leq1+\csd(M_1)\leq1+\csd(M'')=\csd(M')$.
\end{proof}

\begin{lemma}\label{lem:startingseq}
Let $M$ be a connected matroid and $\ell$ a positive integer.
Assume there exists a sequence of $1+\ell$ matroids $M_0=M$ and $M_1,\ldots,M_\ell$ on the same ground set $E(M)$ such that
\begin{itemize}
\item for each $i\in[\ell]$, the matroid $M_i$ is a c$^*$-transformation of $M_{i-1}$ -- that is, 
there exists a matroid $M^+_i$ on the ground set $E(M)\cup\{e_i\}$ such that $M_{i-1}=M^+_i\setminus e_i$ and $M_i=M^+_i\contract e_i$,
\item $M_\ell$ is not connected.
\end{itemize}
If $C\subseteq E(M)$ is such that $\lambda_M(C)>0$ and $\lambda_{M_\ell}(C)=0$, then there is a sequence of matroids
$M'_0=M\contract C$ and $M'_1,\ldots,M'_m$ where $m\leq\ell-\lambda_{M}(C)$, such that
\begin{enumerate}[(i)]
\item\label{it:xstep} for all $i\in[m]$, the matroid $M'_i$ is a c$^*$-transformation of $M'_{i-1}$, and
\item\label{it:xfinal} $M'_m=M_\ell\contract C$.
\end{enumerate}
\end{lemma}

\begin{proof}
We first define another sequence of matroids as follows; $N'_0=M\contract C$, and for $i=1,\ldots,\ell$ we set
$N^+_{i}=M^+_i\contract C$ and $N'_{i}=N^+_{i}\contract e_i$.
Then, naturally, $N'_\ell=M^+_\ell\contract C\contract e_\ell=M^+_\ell\contract e_\ell\contract C=M_\ell\contract C$,
and for all $i\in[\ell]$,~ $N^+_{i}\setminus e_i=M^+_i\contract C\setminus e_i=M^+_i\setminus e_i\contract C=M_{i-1}\contract C=N'_{i-1}$.
So, in particular, each $N'_i$ is a c$^*$-transformation of~$N'_{i-1}$.

We are almost there, only a small problem remains -- the sequence $N'_0,\ldots,N'_\ell$ is too long compared to the desired sequence $M'_0,\ldots,M'_m$.
However, we have $\lambda_{M_\ell}(C)=0$, $\>\lambda_{M_{i-1}}(C)-\lambda_{M_{i}}(C)\leq1$ for each $i\in[\ell]$
(since only one contraction happens in a c$^*$-transformation), and $\lambda_{M_{i}}(C)<\lambda_{M_{i-1}}(C)$ implies that $e_i\in\clo_{M^+_i}(C)$.
So, there are at least $\lambda_M(C)>0$ values of the index $i\in[\ell]$ such that $e_i\in\clo_{M^+_i}(C)$,
and in all such cases $e_i$ is a loop in $N^+_{i}=M^+_i\contract C$, which means that $N'_{i}=N^+_{i}\contract e_i=N^+_{i}\setminus e_i=N'_{i-1}$.
By skipping such repeated matroids from the sequence $N'_0,\ldots,N'_\ell$ we hence get a sequence of desired length $m\leq\ell-\lambda_{M}(C)$.
\end{proof}

\begin{lemma}\label{lem:startingguts}
Let $M$ be a matroid on more than one element.
Then there exists a bipartition $(A,B)$ of $E(M)$ such that $A\not=\emptyset\not=B$ and $\>\max\{\csd(M\contract A),\csd(M\contract B)\}\leq\csd(M)-\lambda_{M}(A)$.
\end{lemma}

\begin{proof}
If $M$ is not connected, then we choose $\emptyset\not=A\subsetneq E(M)$ as the ground set of any component of $M$, and set $B=E(M)-A$.
Then $M$ is a direct sum of $M\restriction\! A=M\contract B$ and of $M\restriction\! B=M\contract A$,
and $\csd(M)=\max\{\csd(M\contract A),\csd(M\contract B)\}$ by definition.
Moreover, $\lambda_{M}(A)=0$, and so $\csd(M)\leq\csd(M)-\lambda_{M}(A)$ holds.
We may thus assume that $M$ is connected.

By \Cref{def:eight-parameters} applied to $\csd(M)$ (informally, by untangling the recursion), we get that there exist an integer $\ell\geq1$
and a sequence of $1+\ell$ matroids $M_0=M$ and $M_1,\ldots,M_\ell$ on the same ground set $E(M)$ which satisfies the assumptions of \Cref{lem:startingseq}
and, moreover, the matroids $M_1,\ldots,M_{\ell-1}$ are connected.
Note that \Cref{def:eight-parameters} also establishes that $\csd(M_{i-1})=\csd(M_{i})+1$ for all $i\in[\ell]$,
and hence $\csd(M)=\csd(M_{\ell})+\ell$.

We now choose $\emptyset\not=A\subsetneq E(M)$ as the ground set of any component of $M_\ell$, and again set $B=E(M)-A$.
So, $\lambda_{M_\ell}(A)=\lambda_{M_\ell}(B)=0$ and $M_\ell$ is a direct sum of $M_\ell\restriction\! A=M\contract B$ and of $M_\ell\restriction\! B=M\contract A$.
In particular, $\csd(M_\ell)=\max\{\csd(M\contract A),\csd(M\contract B)\}$.
Since $M$ is connected, we also have $\lambda_{M}(A)=\lambda_{M}(B)>0$.

Applying \Cref{lem:startingseq} with $C=A$, we get a sequence of matroids $M'_0=M\contract A$ and $M'_1,\ldots,M'_m$ where 
$m\leq\ell-\lambda_{M}(A)$, satisfying the conditions (i) and (ii) of \Cref{lem:startingseq}.
This sequence, by \Cref{def:eight-parameters}, certifies that 
$\csd(M\contract A)\leq \csd(M'_m)+m=\csd(M_\ell\contract A)+m\leq\csd(M_\ell)+m=\csd(M)-\ell+m\leq\csd(M)-\lambda_{M}(A)$,
where $\csd(M_\ell\contract A)\leq\csd(M_\ell)$ follows from \Cref{lem:extcd-to-csd}.
Symmetrically with $C=B$, we also get $\csd(M\contract B)\leq\csd(M)-\lambda_{M}(B)=\csd(M)-\lambda_{M}(A)$, which finishes the proof.
\end{proof}

As observed above, \Cref{def:csd-general} of contraction$^*$-depth stands out
from the other parameter definitions.
We now therefore proceed with showing how exactly \Cref{def:csd-general} relates to \Cref{def:eight-parameters}, particularly to the notion of c$^*$-depth.
This relation, already stated as \Cref{thm:contractionstarx} above, is indirectly captured by the following result of Bria\'nski, Kr{\'{a}}l', and Lamaison~\cite{BriKL25}:

\begin{theorem}[\cite{BriKL25}]\label{thm:contrclosure}
Let $M$ be a matroid, $k$ be the contraction$^*$-depth of $M$, and $\ell$ denote the minimum c-depth of a matroid $M'$ that contains $M$ as a restriction.
Then $k=\ell-1$, unless $M$ is of positive rank and consists of only loops and coloops, in which case~$k=\ell=1$.
\end{theorem}

Note that `$M$ is of positive rank and consists of only loops and coloops' is the negation of
`some element of $M$ is neither a loop nor a coloop, or all elements of $M$ are loops'.
We complement \Cref{thm:contrclosure} with \Cref{lem:extcd-to-csd} and the following natural lemma.

\begin{lemma}\label{lem:csd-to-kklm}
Let $M$ be a matroid and $k$ be the contraction$^*$-depth of $M$.
If some element of $M$ is neither a loop nor a coloop, or all elements of $M$ are loops, then~$k\leq\csd(M)-1$.
\end{lemma}

\begin{proof}
If all elements of $M$ are loops, then $k=0$ and $\csd(M)=1$ by definition.
Otherwise, $M$ has at least two elements (since one-element matroid is either a loop or a coloop) which are not all loops.
We proceed in the proof by induction on $|E(M)|$.

By \Cref{lem:startingguts} there is a bipartition $(A,B)$ of $E(M)$ such that $\max\{\csd(M\contract A),\csd(M\contract B)\}\leq\csd(M)-\lambda_{M}(A)$.
Moreover, if $M$ contained a coloop, we could as well assume that $A$ is the set of all coloops of~$M$ (then $\lambda_{M}(A)=0$),
and we leave this special case to the end of the proof.
Therefore, both matroids $M_1=M\contract A$ and $M_2=M\contract B$ satisfy the assumptions of our lemma and, by induction,
there exist contraction$^*$-depth decompositions $(T^i,f_i)$ of $M_i$ for $i\in[2]$ such that $\hght(T^i)\leq\csd(M_i)-1$.
We construct a tree $T$ by first identifying the roots of $T^1$ and $T^2$, and then adding a path of length $\lambda_{M}(A)$ to the new root of~$T$.
The mapping $f$ of the new decomposition $(T,f)$ of $M$ is simply a union of $f_1$ and $f_2$.
Clearly, $\hght(T)\leq\csd(M_i)-1+\lambda_{M}(A)\leq\csd(M)-1$ for each $i \in [2]$.

It remains to prove that $(T,f)$ is a valid contraction$^*$-depth decomposition (recall \Cref{def:csd-general}).
We have $|E(T)|=|E(T^1)|+|E(T^2)|+\lambda_{M}(A)=\rnk(M_1)+\rnk(M_2)+\big(\rnk_M(A)+\rnk_M(B)-\rnk(M)\big)=
  \big(\rnk(M_1)+\rnk_M(A)\big)+\big(\rnk(M_2)+\rnk_M(B)\big)-\rnk(M)=2\rnk(M)-\rnk(M)=\rnk(M)$.
Moreover, for any $X\subseteq E(M)$, we get by submodularity that $\rnk_M(X)\leq \rnk_M(X\cup A)+\rnk_M(X\cup B)-\rnk(M)$, and
since $M_1=M\contract A$, we have $\rnk_{M_1}(X\cap B)=\rnk_M(X\cup A)-\rnk_M(A)$.
Symmetrically, $\rnk_{M_2}(X\cap A)=\rnk_M(X\cup B)-\rnk_M(B)$, and we altogether derive
\begin{align*}
\rnk_M(X) &\leq \big(\rnk_{M_1}(X\cap B)+\rnk_M(A)\big)+\big(\rnk_{M_2}(X\cap A)+\rnk_M(B)\big)-\rnk(M)
\\  &= \rnk_{M_1}(X\cap B)+\rnk_{M_2}(X\cap A)+\lambda_{M}(A)
.\end{align*}
By induction, $\rnk_{M_1}(X\cap B)\leq|E(T^1_{X\cap B})|$ and $\rnk_{M_2}(X\cap A)\leq|E(T^2_{X\cap A})|$,
and hence, as required, $\rnk_M(X)\leq|E(T^1_{X\cap B})|+|E(T^2_{X\cap A})|+\lambda_{M}(A)=|E(T_X)|$.

Finally, in the special case of $A$ formed by all coloops of~$M$, we get that some element of $B$ is not a loop.
We can hence again get, by induction, a decomposition $(T_1,f_1)$ of $M_1=M\contract A$, where $\hght(T_1)\geq1$.
A valid contraction$^*$-depth decomposition $(T,f)$ of $M$ is then simply constructed by adding a new leaf to the root of $T$ for every element (coloop) of~$A$,
and we get $\hght(T)=\hght(T_1)\leq\csd(M_1)-1=\csd(M)-1$.
\end{proof}

\Cref{thm:contrclosure} with \Cref{lem:extcd-to-csd} and \Cref{lem:csd-to-kklm} now easily imply \Cref{thm:contractionstarx}.

\subsection{Duality and minor-monotonicity}\label{sec:minmonotonicity}

We start with a trivial observation following from the definitions:

\begin{observation}\label{obs:depthsduality}
Let $M$ be a matroid and $M^*$ denote its dual matroid.

a) The measures of branch-depth, cd-depth, and c$^*$d$^*$-depth are all self-dual, that is, 
$\bd(M)=\bd(M^*)$, $\cdd(M)=\cdd(M^*)$, and $\csdsd(M)=\csdsd(M^*)$.

b) The three pairs of measures c-depth and d-depth, c$^*$-depth and d$^*$-depth, and c$^*$d-depth and cd$^*$-depth are
dual to each other, that is, $\cd(M)=\dd(M^*)$, $\csd(M)=\dsd(M^*)$, and $\csdd(M)=\cdsd(M^*)$.
\end{observation}

We also investigate how our depth parameters of \Cref{sec:depths} behave when considering minors.
The established parameter branch-depth is known to be minor-monotone \cite{DeVKO20},
and for the parameter matroid tree-depth, monotonicity can be easily shown from the results of~\cite{matroid-tree-width}.
Regarding the new measures of \Cref{def:eight-parameters}, we summarise the following results.

\begin{proposition}[in part~\cite{DBLP:journals/corr/abs-2402-16215prep}]\label{thm:minor-csdsd}
The parameters c$^*$-depth, d$^*$-depth and c$^*$d$^*$-depth are all minor-monotone.
That is, for every matroid $M$ and every minor $N$ of $M$, we have $\csd(N)\leq\csd(M)$, $\dsd(N)\leq\dsd(M)$ and $\csdsd(N)\leq\csdsd(M)$.
\end{proposition}
\begin{proof}
The first inequality, and the second one by duality, follow from \Cref{lem:extcd-to-csd}.
The inequality $\csdsd(N)\leq\csdsd(M)$ is proved in \cite{DBLP:journals/corr/abs-2402-16215prep},
and we remark that its proof is practically the same as the proof of \Cref{lem:extcd-to-csd}.
\end{proof}

Exact minor-monotonicity does not hold for the remaining five measure of \Cref{def:eight-parameters};
for c-depth, d-depth, and cd-depth, this was already observed by~\cite{DeVKO20}.
For example, the c-depth of the cycle matroid of a cycle with a chord increases when the chord is deleted.
However, it is an easy corollary of other results that c-depth and d-depth are ``functionally minor-monotone'', see upcoming \Cref{cor:fominmon}.

On the other hand, we show by a simple construction that for each of of cd-depth, c$^*$d-depth and cd$^*$-depth,
taking a minor can arbitrarily increase the measure.
Let $C_{i,j}$ denote the graph obtained from the cycle of length $i$ by replacing each edge with $j$ parallel edges, and 
let $D_{i,j}$ be the graph obtained from the cycle of length $i + 1$ by replacing each edge except for one with $j$ parallel edges; see \Cref{fig:fat-cycle} for an illustration.

\begin{figure}[t]
\begin{minipage}{0.45\textwidth}
$$\scalebox{0.75}{
    \begin{tikzpicture}
\def\n{6}
  \def\r{2cm}
  \foreach \i in {1,...,\n}{
    \node[draw, circle, fill=gray, inner sep=1.8pt] (v\i) at ({360/(\n)*(\i-1)}:\r) {};
  }
  \foreach \i in {1,...,\n}{
    \pgfmathtruncatemacro{\j}{mod(\i,\n)+1}
    \draw[bend left=12] (v\i) to (v\j);
    \draw (v\i) to (v\j);
    \draw[bend right=12] (v\i) to (v\j);
    \draw[bend right=24] (v\i) to (v\j);
    \draw[bend left=24] (v\i) to (v\j);
  }
\end{tikzpicture}}$$
\end{minipage} 
\hfill
\begin{minipage}{0.45\textwidth}
$$\scalebox{0.75}{
    \begin{tikzpicture}
\def\n{7}
  \def\r{2cm}
  \foreach \i in {1,...,\n}{
    \node[draw, circle, fill=gray, inner sep=1.8pt] (v\i) at ({360/(\n)*(\i-1)}:\r) {};
  }
  \foreach \i in {2,...,\n}{
    \pgfmathtruncatemacro{\j}{mod(\i,\n)+1}
    \draw[bend left=12] (v\i) to (v\j);
    \draw (v\i) to (v\j);
    \draw[bend right=12] (v\i) to (v\j);
    \draw[bend right=24] (v\i) to (v\j);
    \draw[bend left=24] (v\i) to (v\j);
  }
  \draw (v2)--(v1);
\end{tikzpicture}}$$
\end{minipage} 
    \caption{\textbf{Left:} The ``fat cycle'' $C_{6, 5}$. \textbf{Right:} The graph $D_{6, 5}$. Crucially, by contracting a single edge of $D_{6, 5}$, we obtain $C_{6, 5}$.}
    \label{fig:fat-cycle}
\end{figure}   

\begin{observation}\label{obs:destroy-fat}
For any integers~$i,j\geq2$, $\>\cdd(M(D_{i,j})) \le 3$ and \mbox{$\cdsd(M(C_{i, j})) \le 3$}.
\end{observation}
\begin{proof}
Let $M_1 = M(C_{i, j})$, $M_2 = M(D_{i,j})$, and $f$ denote the unique element of $M_2$ that has no parallel element.
The matroid $N:=M_2\setminus f$ consists of $i$ components, each being a parallel class of $j$ elements, and (separately) contracting one element in each component makes all remaining elements loops.
So, $\cdd(N)\leq2$ by definition, and consequently, $\cdd(M_2)\leq\cdd(N)+1\leq3$.
Furthermore, observe that $M_2\contract f=M_1$ (\Cref{fig:fat-cycle}), and so $N$ is a d$^*$-transformation of the matroid~$M_1$.
Hence, likewise, $\cdsd(M_2)\leq\cdd(N)+1\leq3$.
\end{proof}

To get the desired conclusion in \Cref{pro:nonminorcl}, we show that the matroid $M(C_{i,j})$ has both large cd-depth and large \csddepth.

\begin{lemma}\label{lem:fat-cycles}
Let $i\geq3$ and $j\geq2$ be integers.
\begin{enumerate}[(a)]
\item The cd-depth of the matroid $M(C_{i,j})$ is equal $\min(i,j+2)$.
\item The \csddepth of the matroid $M(C_{2^{i},j})$ is equal $\min(i+1,j+2)$.
\end{enumerate}
\end{lemma}

\begin{proof}a)
Let $M=M(C_{i,j})$. We easily get $\cdd(M)\leq(i-1)+1=i$ by successively contracting one element of each except one parallel class of~$M$,
and $\cdd(M)\leq j+2$ by deleting all elements of one parallel class of~$M$ and finishing as with the matroid $N$ in the proof of \Cref{obs:destroy-fat}.

As for a lower bound, we prove the following claim by induction on~$k+\ell$:
\begin{itemize}
\item Let $M_1$ be a matroid containing a circuit $C$ of length $k\geq3$, such that every element of $E(M_1)\setminus C$ is parallel to some element of~$C$
and that the smallest parallel class of $M_1$ has at least $\ell\geq2$ elements. Then $\cdd(M_1)\geq\min(k,\ell+2)$.
\end{itemize}
This claim can be routinely checked for the basis of induction: $\cdd(M_1)\geq3$ for $k=3$, and $\cdd(M_1)\geq4$ for any $k\geq4$ and~$\ell\geq2$.

In general, let $k>4$ and $\ell>2$ (so that~$\min(k,\ell+2)>4$).
By definition, there is a matroid $M_2$ such that $\cdd(M_1)=\cdd(M_2)+1$, and $M_2=M_1\contract e$ or $M_2=M_1\setminus e$ for some $e\in E(M_1)$.
If $M_2=M_1\setminus e$, then $M_2$ satisfies our assumptions with $k$ and $\ell-1$, and so $\cdd(M_1)=\cdd(M_2)+1\geq\min(k,\ell-1+2)+1\geq\min(k,\ell+2)$ by induction.
If $M_2=M_1\contract e$, then $M_2$ consists of a bunch of loops and a matroid $M_2'$ which satisfies our assumptions with $k-1$ and $\ell$,
and $\cdd(M_2')=\cdd(M_2)$ by definition.
Hence, again, $\cdd(M_1)=\cdd(M_2')+1\geq\min(k-1,\ell+2)+1\geq\min(k,\ell+2)$.

\medskip b)
As for the upper bound, we easily get $\csdd(M)\leq j+2$ as in part a).
Let a matroid $M'$ be obtained by a c$^*$-transformation of $M$ defined as follows: in the graph $C_{2^{i},j}$, add a chord joining two opposite vertices of the ``fat cycle'' and contract it.
Then $M'$ consists of two connected components, each being isomorphic to $M(C_{2^{i-1},j})$.
By a simple recursion, we obtain $\csdd(M)\leq i+1$.

For a lower bound, we proceed similarly as above, by induction with the following claim:
\begin{itemize}
\item Let $M_1$ be a matroid containing a circuit $C$ of length at least $2^k$ where $k\geq3$, such that every element of~$C$ is contained in $M_1$ in a parallel class of size at least $\ell\geq2$.
Then $\csdd(M_1)\geq\min(k+1,\ell+2)$.
\end{itemize}
Again, it can be routinely checked that $\csdd(M_1)\geq4$ for any $k\geq3$ and~$\ell\geq2$.

In general, let $k>3$ and $\ell>2$ (so that~$\min(k+1,\ell+2)>4$).
By definition, there is a matroid $M_2$ such that $\csdd(M_1)=\csdd(M_2)+1$, and either $M_2=M_1\setminus e$ for some $e\in E(M_1)$,
or there is a matroid $M_1^+$ such that $M_1=M_1^+\setminus f$ and $M_2=M_1^+\contract f$ for $\{f\}=E(M_1^+)\setminus E(M_1)$ (a c$^*$-transformation of $M_1$).
If $M_2=M_1\setminus e$, then $M_2$ satisfies our assumptions with $k$ and $\ell-1$, and so $\csdd(M_1)=\csdd(M_2)+1\geq\min(k+1,\ell-1+2)+1\geq\min(k+1,\ell+2)$ by induction.

If $M_2=M_1^+\contract f$, then we employ the following simple result in our proof:
\begin{itemize}
\item{} \cite[Lemma~5.6]{DeVKO20}~
If $C$ with $|C|\geq2$ is a circuit of a matroid $M$ and $f\in E(M)$, then $M\contract f$ contains a~circuit of length at least~$|C|/2$.
\end{itemize}
Since $M_1=M_1^+\setminus f$, $\>C$ is a circuit of length $2^k$ in the matroid $M_1^+$, too.
We apply \cite[Lemma~5.6]{DeVKO20} to the matroid which is the restriction of $M_1^+$ to the ground set consisting of $C$, the parallel classes of the elements of $C$, and the element~$f$.
We thus get that $M_2=M_1^+\contract f$ contains a circuit $C_2$ of length at least $2^{k-1}>2$.
Furthermore, by the setup of our application, all elements of $C_2$ come from the parallel classes of some elements of~$C$,
and so every element of $C_2$ is contained in $M_2$ in a parallel class of size at least $\ell$.
By induction, we hence conclude that $\csdd(M_1)=\csdd(M_2)+1\geq\min(k-1+1,\ell+2)+1\geq\min(k+1,\ell+2)$.
\end{proof}

The following conclusion proves that the listed depth measures are not even ``functionally'' minor-monotone.

\begin{proposition}\label{pro:nonminorcl}
For any {\rm$\gamma\in\{$`cd', `c$^*\!$d', `cd$^*$'$\,\}$} and every positive integer $i$,
there exists a matroid $M$ and a minor $N$ of $M$, such that the $\gamma$-depth of $M$ is at most~$3$ and the $\gamma$-depth of $N$ is at least~$i$.
\end{proposition}

\begin{proof}
For cd-depth, we take the matroid $M=M(D_{i,i})$ as above for $i\geq3$, and $N=M(C_{i,i})$ which is a minor of $M$ (see \Cref{fig:fat-cycle}).
We have $\cdd(M)\leq3$ by \Cref{obs:destroy-fat} and $\cdd(N)=i$ by \Cref{lem:fat-cycles}.
For c$^*\!$d-depth, we achieve the conclusion analogously when taking $M=M(D_{2^i,i})$ and $N=M(C_{2^i,i})$,
and for cd$^*$-depth the conclusion follows by duality and \Cref{obs:depthsduality}.
\end{proof}

Using similar means, we can also complete \Cref{obs:depthsduality} with the following conclusion:

\begin{proposition}\label{pro:nonselfdual}
For each {\rm$\gamma\in\{$`c', `d', `c$^*$', `d$^*$', `c$^*\!$d', `cd$^*$'$\,\}$}, the $\gamma$-depth of a matroid $M$ is functionally incomparable with the $\gamma$-depth of the dual matroid $M^*$.
\end{proposition}

\begin{proof}
Regarding c-depth (and dually also d-depth), we consider the matroid $M$ consisting of only $i\geq2$ parallel elements.
Then $\cd(M)\leq2$ since the first contraction leaves only a bunch of loops.
On the other hand, the dual matroid $M^*$ is a circuit, and since contracting any at most $i-1$ elements leaves $M^*$ connected, we get $\cd(M^*)\geq i$.
Therefore, there is no function $f$ such that $\cd(M^*)\leq f(\cd(M))$ holds for all matroids~$M$.

The same conclusion applies to c$^*$-depth (and dually also to d$^*$-depth) since, by existing results (\cite{KarKLM17,DeVKO20}, see \Cref{cor:xd_circuits}),
c$^*$-depth is functionally related to c-depth.

Finally, regarding cd$^*$-depth (and dually c$^*\!$d-depth), we consider the matroid $M$ equal to $M=M(C_{2^i,i})$ for $i\geq3$.
By \Cref{obs:destroy-fat}, $\cdsd(M)\leq3$, by \Cref{obs:depthsduality}, $\cdsd(M^*)=\csdd(M)$, and by \Cref{lem:fat-cycles}, $\csdd(M)=i+1$.
Again, there cannot be a function $f$ such that $\cdsd(M^*)=\csdd(M)\leq f(\cdsd(M))$ holds for all matroids~$M$.
\end{proof}

\subsection{Comparing the depth parameters}

The goal of this section is to compare ``strengths'' of the parameters from \Cref{sec:depths}, and in particular that of the eight variants of \Cref{def:eight-parameters}.
This comparison is illustrated in \Cref{fig:hasse-csdsd} above, while full references and formal proofs of the mutual relations are provided in this section.
As a first step, we summarise some known and useful facts.

For a matroid $M$, let $u(M)$ (as `circUmference') denote the size of the longest circuit in~$M$; in case $M$ has no circuit, 
we set $u(M) = 1$. Likewise, let $u^*(M):=u(M^*)$ denote the size of the longest cocircuit in~$M$,
and set $u^*(M) = 1$ if $M$ is of rank zero.
DeVos, Kwon, and Oum~\cite{DeVKO20} showed that the d-depth and the c-depth of matroids are
functionally equivalent to $u^*(M)$ and $u(M)$, respectively. Formally:

\begin{theorem}[\cite{DeVKO20}]\label{thm:cd_circuits}
    If $M$ is a matroid, then $\log_2(u^*(M)) \leq\> \dd(M) \>\leq u^*(M)(u^*(M)+1)/2$, and dually, $\log_2(u(M)) \leq\> \cd(M) \>\leq u(M)(u(M)+1)/2$.
\end{theorem}

Analogously (and independently of \cite{DeVKO20}), Kardo\v{s}, Kr{\'{a}}l', Liebenau, and Mach~\cite{KarKLM17} proved the following relationship for c$^*$-depth.

\begin{theorem}[\cite{KarKLM17} via \Cref{thm:contractionstarx}]\label{thm:csd_circuits}
    Let $M$ be a matroid. Then 
    \begin{align*}
        \log_2(u(M)) \leq \csd(M) \leq u(M)^2+1 .
    \end{align*}
\end{theorem}

Comparing the latter two results and using duality, we can immediately conclude that the following two pairs
of depth parameters are functionally equivalent:

\begin{corollary}\label{cor:xd_circuits}
The following relations hold; $\cd \funeq \csd$ and $\dd \funeq \dsd$.
\end{corollary}

Moreover, since c$^*$-depth is minor-monotone by \Cref{lem:extcd-to-csd}, using \Cref{cor:xd_circuits} and duality
we can also easily complete the full picture of \Cref{sec:minmonotonicity}:

\begin{corollary}\label{cor:fominmon}
For every matroid $M$ and every minor $N$ of $M$, we have $\cd(N)\funle\cd(M)$ and $\dd(N)\funle\dd(M)$.
\end{corollary}

Observe that we trivially have $\cd(M)\leq\cdd(M)$ and $\dd(M)\leq\cdd(M)$ from the definition.
Next, we compare the non-starred and starred variants of the remaining measures from \Cref{def:eight-parameters},
which is easy (cf.~\Cref{rem:starstronger}) but not directly trivial:

\begin{lemma}\label{lem:starredineq}
For every matroid $M$, we have $\csdsd(M)\leq\csdd(M)\leq\cdd(M)$ and $\csdsd(M)\leq\cdsd(M)\leq\cdd(M)$.
\end{lemma}

\begin{proof}
Let $M$ be a matroid.
The proofs for all four claimed inequalities are very similar, so we demonstrate only the proof for $\csdd(M)\leq\cdd(M)$.
We proceed by induction on $|E(M)|$. 
If $|E(M)| \le 1$, then $\csdd(M) = 1 = \cdd(M)$ by \Cref{def:eight-parameters}\,(\ref{it:one-element}).
From now on, suppose that $|E(M)| > 1$. 

First, suppose that $M$ is disconnected.
By \Cref{def:eight-parameters}\,(\ref{it:bycomponents}), there is a connected component $M'$ of $M$ such that $\csdd(M) = \csdd(M')$.
Now observe that $\csdd(M') \le \cdd(M') \le \cdd(M)$ as required; the first inequality follows by induction hypothesis, and the second one again by \Cref{def:eight-parameters}\,(\ref{it:bycomponents}).

Now suppose that $M$ is connected.
Let $M'$ be a matroid such that $M'$ is cd-transformed from $M$ and $\cdd(M) = \cdd(M') + 1$, as in \Cref{def:eight-parameters}\,(\ref{it:oneredstep}).
If $M' = M \setminus e$ for some $e \in E(M)$, then $M'$ is c$^*$d-transformed from $M$, and $\csdd(M) \le \csdd(M') + 1 \le \cdd(M') + 1$ as required;
the first inequality follows by \Cref{def:eight-parameters}\,(\ref{it:oneredstep}~b) and the second one by induction~hypothesis.

The other possible case is $M' = M \contract e$ for some $e \in E(M)$.
Let $M^+$ be a matroid with $f \in E(M^+)$ such that $M^+ \setminus f = M$ and $f$ is parallel to $e$ in $M^+$.
Observe that $e$ is a loop in the matroid $M^- := M^+ \contract f$, and $M^-\setminus e=M'$.
By \Cref{def:eight-parameters}\,(\ref{it:bycomponents}), we get $\csdd(M^-)=\csdd(M^-\setminus e)$ and $\cdd(M^-)=\cdd(M^-\setminus e)=\cdd(M')$,
and $\csdd(M^- \setminus e) \le \cdd(M^- \setminus e)$ holds by induction.
Moreover, $\csdd(M) \le \csdd(M^-) + 1$ because $M^-$ is c$^*$d-transformed from $M$, and hence
$\csdd(M) \le \csdd(M^-) + 1 \le \cdd(M^-) + 1 = \cdd(M') +1 = \cdd(M)$.
\end{proof}

Very recently, Bria\'nski, Hlin\v{e}n\'y, Kr{\'{a}}l', and Pek{\'{a}}rkov{\'{a}}~\cite{DBLP:journals/corr/abs-2402-16215prep} proved the following,
in fact quite nontrivial, inequality.

\begin{theorem}[\cite{DBLP:journals/corr/abs-2402-16215prep}]\label{thm:bd-csdsd}
For every matroid $M$, we have $\bd(M)\leq\csdsd(M)\leq 2\bd(M)^2+1$.
\end{theorem}

Now we proceed to prove incomparability results for pairs of measures.
To start, we observe that c-depth and d-depth are functionally incomparable.

\begin{observation}\label{lem:incom-cddd}
$\cd \funnleq \dd$ and $\dd \funnleq \cd$.
\end{observation}
\begin{proof}
By duality, it is enough to prove $\cd \funnleq \dd$.
Let $M_i$ be the cycle matroid of a graph cycle of length~$i\geq3$. Then trivially~$\dd(M_i)\leq2$.
On the other hand, since contracting any at most $i-1$ elements leaves $M_i$ connected, we get $\cd(M_i)\geq i$.
\end{proof}

To show that \csddepth and \cdsdepth are functionally incomparable, we use the example matroids $M(C_{i,j})$ and $M(D_{i,j})$ constructed in \Cref{sec:minmonotonicity}.

\begin{proposition}\label{pro:allstrict}
\begin{enumerate}[(a)]
\item $\cdsd\funnleq\csdd$ and $\csdd\funnleq\cdsd$.
\item All arrow relations depicted in \Cref{fig:hasse-csdsd} are strict.
\end{enumerate}
\end{proposition}

\begin{proof}a) By using duality, it is again enough to prove that $\csdd\funnleq\cdsd$.
Similarly as in the proof of \Cref{pro:nonselfdual}, we consider the matroid $M$ equal to $M=M(C_{2^i,i})$ for arbitrary $i\geq3$.
By \Cref{obs:destroy-fat}, $\cdsd(M)\leq3$, by \Cref{obs:depthsduality}, $\cdsd(M^*)=\csdd(M)$, and by \Cref{lem:fat-cycles}, $\csdd(M)=i+1$.
Hence, $\csdd(M)$ cannot be bounded by a function of $\cdsd(M)$ for arbitary~$M$.

b)
Suppose, for the sake of a contradiction, that $\cd \funle \cdd$. Then, it follows from duality that $\dd \funle \cdd$.
Since $\cdd \funle \cd$ and $\cdd \funle \dd$ hold trivially, we obtain $\cd\funeq\cdd\funeq\dd$ which contradicts \Cref{lem:incom-cddd}.
Hence, $\cd\funnleq\cdd$ and dually, $\dd\funnleq\cdd$.
The proofs of $\cdd \funnleq \csdd$, $\cdd \funnleq \cdsd$, $\csdd\funnleq\csdsd$, and $\cdsd\funnleq\csdsd$ use \Cref{pro:allstrict}\,(a)
and are otherwise completely analogous to the proof of $\cd\funnleq\cdd$, so we omit them here.
\end{proof}

We are now ready to summarise all relations between the parameters.

\begin{proposition}\label{pro:allvalid}
All relations depicted in \Cref{fig:hasse-csdsd} between the branch-depth and the eight depth parameter variants of \Cref{def:eight-parameters}
are valid in the class of all matroids.
\end{proposition}
\begin{proof}
The statement follows immediately from \Cref{obs:depthsduality}, \Cref{cor:xd_circuits}, \Cref{lem:starredineq}, \Cref{thm:bd-csdsd}, \Cref{lem:incom-cddd}, and \Cref{pro:allstrict}.
\end{proof}

Looking at \Cref{fig:hasse-csdsd}, one might ask whether, perhaps, cd-depth is functionally related to the ``intersection of c$^*$d-depth and cd$^*$-depth'',
that is, whether bounding both c$^*$d-depth and cd$^*$-depth of a matroid bounds also its cd-depth. We show a negative answer.

\begin{figure}[t]
$$\begin{tikzpicture}[scale=1.2]
\def\n{7}\def\nn{8}
\def\r{2cm}\def\rr{22.7mm}
  \foreach \i in {1,...,\n}{
    \node[draw, circle, fill=gray, inner sep=1.5pt] (v\i) at ({360/(\n)*(\i-1)}:\r) {};
  }
  \foreach \i in {1,2,3,5,6,...,\n}{
    \pgfmathtruncatemacro{\j}{mod(\i,\n)+1}
    \draw[bend left=12] (v\i) to (v\j);
    \draw (v\i) to (v\j);
    \draw[bend right=12] (v\i) to (v\j);
    \draw[bend right=24] (v\i) to (v\j);
    \draw[bend right=36] (v\i) to (v\j);
    \draw[bend left=24] (v\i) to (v\j);
  }
  \draw (v4)--(v5);
  \foreach \i in {1,...,\nn}{
    \node[draw, circle, fill=gray, inner sep=1.5pt] (w\i) at ($ ({360/(2*\nn)+360/(\nn)*(\i-1)}:\rr)+(-3.9,0) $) {};
  }
  \foreach \i in {1,2,...,\nn}{
    \pgfmathtruncatemacro{\j}{mod(\i,\nn)+1}
    \draw (w\i) to (w\j);
  }
  \foreach \rx in {12,14,16,18,20}{
    \foreach \i in {3,...,\n}{
      \node[draw, circle, fill=gray, inner sep=1.5pt] (w\i\rx) at ($ ({360/(2*\nn)+360/(\nn)*(\i-1)}:\rx mm)+(-3.9,0) $) {};
    }
    \draw (w3\rx)--(w2) (w\n\rx)--(w\nn);
    \foreach \i in {4,...,\n}{
      \pgfmathtruncatemacro{\j}{mod(\i+\nn-1,\nn)}
      \draw (w\i\rx)--(w\j\rx);
    }
  }
  \node[draw=none] at (-1.6,0) {$e$};
  \node[draw=none] at (-2,-0.7) {$u$};
  \node[draw=none] at (-2.5,1.2) {$f$};
  \node[draw=none] at (-2.8,2.15) {$v$};
\end{tikzpicture}$$
    \caption{The planar graph $H_k$ for $k=6$ whose cycle matroid is $L_k=M(H_k)$. Observe that the planar dual of $H_k$ is isomorphic to $H_k$, taking the edge $e$ to the dual edge of~$f$.}
    \label{fig:fat-cycle-dcycle}
\end{figure}   

Let $H_k$ for $k\geq2$ denote the (planar) graph illustrated in \Cref{fig:fat-cycle-dcycle}.
The graph $H_k$, on $(k+2)+k(k-1)=k^2+2$ vertices and with two specific edges $e,f\in E(H_k)$, is constructed as follows.
Start with the union of a cycle $C_1$ of length $k+1$ and a cycle $C_2$ of length $k+2$ such that $E(C_1)\cap E(C_2)=\{e\}$ and $f\in E(C_2)$ is adjacent to~$e$.
Then replace every edge of $E(C_1)\setminus\{e\}$ with $k$ parallel edges, and multiply the length-$k$ path $C_2-\{e,f\}$ with ends $u,v$ into $k$ internally disjoint paths with the same ends~$u,v$.
Finally, let $L_k=M(H_k)$ be the cycle matroid of this graph. Note that the dual matroid $L_k^*$ is isomorphic to $L_k$ and this isomorphism ``swaps'' $e$ with~$f$ (\Cref{fig:fat-cycle-dcycle}).

\begin{proposition}\label{pro:fat-cycle-dcycle}
For every $k\geq2$:
\begin{enumerate}[a)]
\item $\csdd(L_k)\leq4$ and $\cdsd(L_k)\leq4$, and
\item $\cdd(L_k)>k$.
\end{enumerate}
\end{proposition}

\begin{proof}
a) The matroid $L_k\contract e$ has two components $L_k'$ and $L_k''$, where $L_k'$ is isomorphic to the matroid $M(C_{k,k})$ from \Cref{obs:destroy-fat}, and so $\cdsd(L_k') \le 3$.
Further, the matroid $L_k''\contract f$ consists of $k$ components, each being a circuit, and so $\cdsd(L_k''\contract f) \le \dd(L_k''\contract f) \le 2$.
Altogether, $\cdsd(L_k)\leq4$, and $\csdd(L_k)=\cdsd(L_k^*)\leq4$ follows by self-duality.

b) We proceed similarly as in the proof of \Cref{lem:fat-cycles}.
Before we start, we remark that while cd-depth is far from being monotone under taking minors (\Cref{pro:nonminorcl}),
it is monotone in the following easy special subcase:
\begin{itemize}
\item If $e,f\in E(M)$ are parallel elements in a matroid $M$, then $\cdd(M\setminus e)\leq\cdd(M)$.
\end{itemize}

Let, for $\ell<k$, $\ca L_{k,\ell}$ denote the class of all matroids $L_k^-$ such that $L_k^-$ results from $L_k$ by iteratively deleting or contracting at most $\ell$ elements,
and after that, removing all resulting loops and coloops. Moreover, it is required that $e\in E(L_k^-)$ or $e$ has been deleted, and $f\in E(L_k^-)$ or $f$ has been contracted.
It is easy to observe that all members of $\ca L_{k,\ell}$ are connected matroids.
We are going to prove, by induction on $k-\ell$, that $\cdd(M)>k-\ell$ holds for all $M\in\ca L_{k,\ell}$.
For $\ell=0$, this gives the desired conclusion.

When $M\in\ca L_{k,\ell}$ and $k-\ell=1$, we trivially have $\cdd(M)>1$ since $M$ is connected with more than one element.
Let $k-\ell>1$. By definition, and up to duality, we have that $\cdd(M)=\cdd(M\contract e')+1$ for some $e'\in E(M)$.
Then $M\contract e'\in\ca L_{k,\ell+1}$ and the statement holds by induction, unless $e'=e$.
If $e'=e$, then the matroid $M_1=M\contract e'$ contains a component $M_1'$ such that (possible) deletion of selected parallel elements yields 
a restriction $M_1''$ isomorphic to $M(C_{i,k-\ell})$ where $i\geq k-\ell$.
By \Cref{lem:fat-cycles}, we then get $\cdd(M_1')\geq\cdd(M_1'')\geq k-\ell$ and $\cdd(M)=\cdd(M_1)+1\geq\cdd(M_1')+1>k-\ell$.
\end{proof}

\subsection{Relation to matroid tree-depth}\label{sec:mtreedepth}

In this section, we study the new notion of matroid tree-depth (defined in \Cref{sec:treede})
and compare it to other measures studied in this paper, in particular to c$^*$-depth. 
This completes the picture of relations between parameters depicted in \Cref{fig:hasse-csdsd}.

Our main result in this direction is the following theorem.
\begin{theorem}\label{thm:mtreedepth-rel}
For every matroid $M$, $\>\mtd(M)\leq\csd(M)\leq\mtd(M)^2+1$.
\end{theorem}
\Cref{thm:mtreedepth-rel} follows from Lemmas~\ref{lem:mtreedepth-leq} and~\ref{lem:mtreedepth-geq}, which will be established later in this section.

Recalling the function $\omega_M$ of a matroid $M$ from \Cref{sec:treede}, 
defined by $\omega_M(X_1,\ldots,X_k)=\sum_{i=1}^k \rnk_M(E(M)-X_i)\>-(k-1)\rnk(M)$, we first prove:

\begin{lemma}\label{lem:omegaineq}
Let $M$ be a matroid and $X_1,\ldots,X_k\subseteq E(M)$ a collection of $k\geq1$ pairwise disjoint sets.
Assume that a matroid $M'\not=M$ is a c$^*$-transformation of~$M$, that is, $M=M^+\setminus e$ and $M'=M^+\contract e$ for some matroid $M^+$
and $e\in E(M^+)$ which is not a loop nor coloop.
Then the following hold:

a) $\omega_M(X_1,\ldots,X_k)\leq\omega_{M'}(X_1,\ldots,X_k)+1$;

b) if $e\in\cl_{M^+}(E(M)-X_i)$ for all $i\in[k]$, then $\omega_M(X_1,\ldots,X_k)=\omega_{M'}(X_1,\ldots,X_k)+1$;

c) if $e\in\cl_{M^+}(X_i)$ for some $i\in[k]$, then $\omega_M(X_1,\ldots,X_k)\geq\omega_{M'}(X_1,\ldots,X_k)$.
\end{lemma}
\begin{proof}
a) If $M'$ is a c$^*$-transformation of $M$, then since $e$ is not a loop nor coloop, we have $\rnk(M)=\rnk(M^+)=\rnk(M') + 1$.
Moreover, for all $i\in[k]$, $\>\rnk_M(E(M)-X_i)=\rnk_{M^+}(E(M)-X_i)$, and $\big(\rnk_{M^+}(E(M)-X_i)-\rnk_{M'}(E(M)-X_i)\big)\in\{0,1\}$ since $M'=M^+\contract e$.
Therefore, $\omega_M(X_1,\ldots,X_k)-\omega_{M'}(X_1,\ldots,X_k)\leq k-(k-1)=1$.

b) In this case, $\rnk_M(E(M)-X_i)=\rnk_{M^+}(E(M)-X_i)=\rnk_{M'}(E(M)-X_i)+1$ for all $i\in[k]$, and so
$\omega_M(X_1,\ldots,X_k)-\omega_{M'}(X_1,\ldots,X_k)=\sum_{i=1}^k\big(\rnk_M(E(M)-X_i)-\rnk_{M'}(E(M)-X_i)\big)-(k-1)\big(\rnk(M)-\rnk(M')\big)=k-(k-1)=1$.

c) Since $e\in\cl_{M^+}(X_i)$ and the sets $X_1,\ldots,X_k$ are disjoint, for all $j\in[k]$, $j\not=i$, we have $e\in\cl_{M^+}(E(M)-X_j)$.
Consequently, $\rnk_M(E(M)-X_j)=\rnk_{M^+}(E(M)-X_j)=\rnk_{M'}(E(M)-X_j)+1$, and we conclude
$\omega_M(X_1,\ldots,X_k)-\omega_{M'}(X_1,\ldots,X_k)\geq(k-1)-(k-1)=0$.
\end{proof}

Before proceeding, we recall the following result, which will be used in the proofs.

\begin{theorem}[Geelen, Gerards, and Whittle~{\cite[Theorem 6.1 and Claim 6.3.1]{DBLP:journals/siamdm/GeelenGW06}}]
\label{thm:GGWadde}
Let $M$ be a matroid and let $(A,B)$ be a connected bispan in $M$. 
Then there exists a matroid $N$ on the ground set $E(M)\cup\{e\}$ which is a relatively free extension of $M$ by $e$ in $(A,B)$.
In particular, $e$ is not a loop and $e\in\clo_{N}(A)\cap\clo_{N}(B)$.
\end{theorem}

\begin{lemma}\label{lem:mtreedepth-leq}
For every matroid $M$, $\>\mtd(M)\leq\csd(M)$.
\end{lemma}

\begin{proof}
For this proof, we say that a matroid tree-decomposition $(T, \tau)$ is a $(t,d)$-decomposition if its width 
is at most $t$ and the radius of $T$ is at most~$d$.
We are going to prove, by induction, that every matroid $M$ admits a $(t,d)$-decomposition for $t=\csd(M)$ and $d=t-1$ if $M$ is connected,
while $d=t$ if $M$ is disconnected.

The claim trivially holds if $\csd(M)=1$; then we have $t=\csd(M)=1$ and~$d=0$ (if $M$ connected, i.e., $M$ has at most one element) or~$d=1$ (otherwise).

Assume that the claim holds for all matroids of c$^*$-depth less than $t=\csd(M)$, as well as for all matroids of c$^*$-depth equal to~$t$ which are smaller than~$M$.
If $M$ is not connected, then we inductively obtain $(t,t-1)$-decompositions of the components of $M$, 
pick a root in each one which has distance at most $t-1$ to every vertex in it, and construct a $(t,t)$-decomposition of~$M$ by joining the said roots to a common new vertex.
Otherwise, by \Cref{def:eight-parameters}, let $M'$ be a c$^*$-transformation of the matroid $M$ such that $\cd(M')=\cd(M)-1=t-1$.
Then, by induction, we obtain a $(t-1,t-1)$-decomposition $(T, \tau)$ of~$M'$, and we have that $(T, \tau)$ is at the same time 
a $(t,t-1)$-decomposition of~$M$ by \Cref{def:mtreewidth} and \Cref{lem:omegaineq}\,a).
\end{proof}

\begin{lemma}\label{lem:mtreedepth-geq}
If a matroid $M$ has a matroid tree-decomposition of width $t$ and radius~$d$, then $\csd(M)\leq t\cdot d+1$.
\end{lemma}

\begin{proof}
Let $(T, \tau)$ denote the assumed decomposition, and let $r\in V(T)$ be a vertex with distance at most $d$ from every vertex of $T$.
We proceed by induction on the radius $d$ of~$T$.
If $d=0$, then, trivially, $t=\rnk(M)$ and $\csd(M)\leq \rnk(M)+1=t\cdot1+1$.
We hence assume that $d>0$ and the vertex $r$ is of degree~$k$,
and we denote by $X_1,\ldots,X_k$ the (disjoint) subsets of $E(M)$ mapped by $\tau$ to the vertices of the $k$ components of~$T-r$.
Let $X_0=\tau^{-1}(r)=E(M)-(X_1\cup\ldots\cup X_k)$.

Let $a=\omega_M(X_1,\ldots,X_k)\leq t$ and $M_0=M$.
For $i\geq1$, we inductively construct a matroid $M_i$ as a c$^*$-transformation of $M_{i-1}$ such that $\omega_{M_i}(X_1,\ldots,X_k)=a-i$, as follows.
If $\rnk_{M_{i-1}}(X_0)>0$, then we construct $M^+_i$ by adding an element $e_i$ in parallel to any non-loop element of~$X_0$, and $M_i=M^+_i\contract e_i$.
Otherwise, if there is $j\in[k]$ such that $\lambda_{M_{i-1}}(X_j)>0$, we do the following.
We construct $M^+_i$ as a relatively free extension of $M_{i-1}$ by $e_i$ in $(X_j,E(M)-X_j)$, cf.~\Cref{thm:GGWadde}, and again $M_i=M^+_i\contract e_i$.
In both cases, we conclude with $\omega_{M_{i}}(X_1,\ldots,X_k)=\omega_{M_{i-1}}(X_1,\ldots,X_k)-1$ by \Cref{lem:omegaineq}\,b).

We stop when $\rnk_{M_{i}}(X_0)=0$ and $\lambda_{M_{i-1}}(X_j)=0$ for $j\in[k]$.
In such case $\rnk_{M_{i}}(X_0)=0$ implies $\sum_{j=1}^k \rnk_{M_{i}}(X_j)\geq\rnk(M_{i})$,
and so we get $\omega_{M_{i}}(X_1,\ldots,X_k)= \sum_{j=1}^k \rnk_{M_{i}}(E(M)-X_j)
 -(k-1)\rnk(M_{i})\leq \sum_{j=1}^k \rnk_{M_{i}}(E(M)-X_j)+\sum_{j=1}^k \rnk_{M_{i}}(X_j)-k\rnk(M_{i}) \leq \sum_{j=1}^k \lambda_{M_{i}}(X_j)=0$.
Thus, by \Cref{lem:omegaMzero}, $\omega_{M_{i}}(X_1,\ldots,X_k)=0$ and so $i=a\leq t$, which means
that $\csd(M)\leq\csd(M_a)+t$ by \Cref{def:eight-parameters}.

For each $j\in[k]$, we denote by $N_j=M_a\restriction\!X_j$ the restriction of $M_a$ to $X_j$.
The corresponding component $T_j$ of $T-r$, and the restriction $\tau_j$ of $\tau$ to $X_j$, define a tree-decomposition $(T_j,\tau_j)$ of $N_j$.
The radius of $T_j$ is at most~$d-1$, and by repeated application of \Cref{lem:omegaineq}\,c) to the construction of $M_a$, 
we get that the width of $(T_j,\tau_j)$ is at most~$t$.
Thus, by induction, $\csd(N_j)\leq t(d-1)+1$ for all $j\in[k]$, then $\csd(M_a)\leq t(d-1)+1$, and we conclude that
$\csd(M)\leq\csd(M_a)+t\leq t(d-1)+1+t=t\cdot d+1$.
\end{proof}

\subsection{Depth parameters of matrices}

In \Cref{sec:depths}, we have surveyed several existing definitions of depth measures that are based on a particular matrix representation of a matroid,
such as \Cref{def:csd-representable}.
Although these definitions appear analogous to the other presented definitions (namely \Cref{def:eight-parameters}), their relation
to the corresponding matroid-based definitions is not clear due to a different domain of the definition(s).

In particular, we recall \Cref{ex:Fannostar} from \Cref{sec:recursivemeas}, which nicely illustrates that it is not at all clear
from the definition whether, e.g., \Cref{def:csd-representable} is invariant on the choice of a matrix representation.
In other words, it is unclear whether the represented c$^*$-depth $\csd(\veAm)$ is a property of the matrix representation $\ve A$ or of the matroid $M(\ve A)$.
We answer this interesting question affirmatively by proving the following.

\begin{proposition}\label{cor:csdmatrixeq}
Let $\veAm$ be a matrix over a field $\FF$ and $M=M(\ve A)$. Then $\csd(\ve A)=\csd(M)$ and $\dsd(\ve A)=\dsd(M)$.
\end{proposition}

\begin{proof}
Observe that for every column vector $\ve v$ considered in \Cref{def:dsd-representable},
the matroid $M(\veAm \extcontr \vev)$ is a c$^*$-transformation of the matroid $M(\ve A)$.
So, immediately, $\csd(M)\leq\csd(\ve A)$.

In the other direction of the inequality, $\csd(\ve A)\leq\csd(M)$, we proceed by induction on $\csd(M)$. 
If $|E(M)|=1$, then the claim trivially holds.
Otherwise, by \Cref{lem:startingguts}, there exists a bipartition $(A,B)$ of $E(M)$ such that
$\max\{\csd(M\contract A),\csd(M\contract B)\}\leq\csd(M)-\lambda_{M}(A)$.
Considering $(A,B)$ as a bipartition of the columns of the matrix $\ve A$, we may pick an arbitrary basis of the subspace
$\langle A\rangle\cap\langle B\rangle$ (which is of cardinality~$\lambda_{M}(A)$ and can be empty if $\lambda_{M}(A)=0$) 
and iteratively add and contract its vectors in $\ve A$.
The resulting matrix $\ve A'$ obviously represents the direct sum $M'$ of $M\contract A$ and $M\contract B$.
By our construction, $\csd(\ve A)\leq\csd(\ve A')+\lambda_{M}(A)$, and by our induction assumption,
$\csd(\ve A')\leq\csd(M')=\max\{\csd(M\contract A),\csd(M\contract B)\}\leq\csd(M)-\lambda_{M}(A)$.
Together, $\csd(\ve A)\leq\csd(M)$, as desired.

The equality $\dsd(\ve A)=\dsd(M)$ now follows by applying the previous proof to the dual matroid~$M^*$ (cf.~\Cref{obs:depthsduality}).
\end{proof}

The same question can be asked about the determinacy of the depth measure $\csdd(\ve A)$ 
(\Cref{def:csdd-representable}) by the underlying matroid $M(\ve A)$ represented by~$\ve A$.
We give an analogous affirmative answer using the following technical generalisation of \Cref{lem:startingguts}.

\begin{lemma}\label{lem:startinggutsd}
Let $M$ be a matroid on more than one element.
Then there exists a tripartition $(A,B,D)$ of $E(M)$ such that $A\not=\emptyset\not=B$ and 
$$\max\{\csdd(M\setminus D\contract A),\> \csdd(M\setminus D\contract B)\}\leq \csdd(M)-\lambda_{M}(A)-|D| .$$
\end{lemma}

\begin{proof}
We start similarly as in the proof of \Cref{lem:startingguts}.
If there is a set $D\subseteq E(M)$ (possibly $D=\emptyset$) such that $M'_0=M\setminus D$ is not connected, and $\csdd(M)=\csdd(M'_0)+|D|$,
then we choose $\emptyset\not=A\subsetneq E(M)-D$ as the ground set of any component of $M'_0$, and set \mbox{$B=E(M)-D-A$}.
Then $M'_0$ is a direct sum of $M'_0\restriction\! A=M'_0\contract B$ and of $M'_0\restriction\! B=M'_0\contract A$,
and so $\csdd(M'_0)=\max\{\csdd(M'_0\contract A),$ $\csdd(M'_0\contract B)\}$.
Moreover, $\lambda_{M'_0}(A)=0$, and hence $\csdd(M'_0)\leq\csdd(M'_0)-\lambda_{M'_0}(A)=\csdd(M)-\lambda_{M}(A)-|D|$ holds trivially.

We may thus assume that $M$ is connected.
Again, by ``untangling'' \Cref{def:eight-parameters} as applied to $\csdd(M)$, we get that there exist an integer $k\geq1$
and a sequence of $1+k$ matroids $M_0=M$ and $M_1,\ldots,M_k$ which satisfies the following;
for each $i\in[k]$, the matroid $M_i$ is a c$^*$-transformation of $M_{i-1}$ or $M_i=M_{i-1}\setminus f$ for some $f\in E(M)$,
and the matroid $M_k$ is disconnected.
Furthermore, due to \Cref{def:eight-parameters}, $\csdd(M_{i-1})=\csdd(M_i)+1$ for $i\in[k]$, and so $\csdd(M_0)=\csdd(M_{k})+k$.

Let $D=E(M)-E(M_k)$, i.e., $D$ is the set of elements which are deleted along the sequence.
We define a new sequence $M'_0,M'_1,\ldots,M'_k$ by $M'_i=M_i\setminus D$ for $0\leq i\leq k$.
Obviously, $M'_k=M_k$ and for $|D|$ values $i\in[k]$ such that $M_i=M_{i-1}\setminus f$ where $f\in D$ we have $M'_i=M'_{i-1}$.
For the remaining $k-|D|$ indices $i\in[k]$ (such that $M'_i\not=M'_{i-1}$), we have that $M'_i$ is a c$^*$-transformation of $M'_{i-1}$,
and so $\csdd(M'_{i-1})\leq\csdd(M'_i)+1$.
Since $\csdd(M_0)=\csdd(M_{k})+k=\csdd(M'_{k})+k$, and we have previously got $\csdd(M_0)\leq\csdd(M'_0)+|D|$ and $\csdd(M'_0)\leq\csdd(M'_k)+(k-|D|)$,
this chain of inequalities must hold as equalities everywhere.
Thus, $\csdd(M_0)=\csdd(M'_0)+|D|$ where $M'_0=M\setminus D$, and $\csdd(M'_{i-1})=\csdd(M'_i)+1$ for all $i\in[k]$ such that $M'_i\not=M'_{i-1}$.

Now, if $M'_0=M\setminus D$ is disconnected, then the situation is as solved in the first paragraph.

We may thus assume that $M\setminus D$ is connected, and recall that $M'_k=M_k$ is disconnected.
Then, by skipping repeated matroids in the sequence $M'_0,\ldots,M'_k$ and stopping at the first disconnected member of it,
we get an integer $1\leq\ell\leq k-|D|$ and a new sequence of matroids $N_0=M'_0=M\setminus D$ and $N_1,\ldots,N_\ell$ 
such that $N_0,\ldots,N_{\ell-1}$ are connected, $N_\ell$ is disconnected, and $N_i$ is a c$^*$-transformation of $N_{i-1}$ for all $i\in[\ell]$.
Furthermore, by the previous arguments, $\csdd(N_{i-1})=\csdd(N_i)+1$ for all $i\in[\ell]$, and so $\csdd(N_0)=\csdd(N_\ell)+\ell$.

From this point, we again continue as in the proof \Cref{lem:startingguts}.
We choose $\emptyset\not=A\subsetneq E(M)$ as the ground set of any component of $N_\ell$ and $B=E(M)-A$.
So, $\lambda_{N_0}(A)>0$ and $\lambda_{N_\ell}(A)=\lambda_{N_\ell}(B)=0$, and
$N_\ell$ is a direct sum of $N_\ell\contract B$ and of $N_\ell\contract A$.

Applying \Cref{lem:startingseq} with the matroid $N_0$ and $C=A$, we get a sequence of matroids $N'_0=N_0\contract A$ and 
$N'_1,\ldots,N'_m=N_\ell\contract A$ where $m\leq\ell-\lambda_{N_0}(A)$, satisfying the conditions (i) and (ii) of \Cref{lem:startingseq}.
This sequence, by \Cref{def:eight-parameters}, certifies that 
\begin{align*}
\csdd(N_0\contract A)&\leq\csdd(N_\ell\contract A)+m \leq\csdd(N_\ell)+m \leq\csdd(N_0)-\ell+m
\\ &=\csdd(M_0)-|D|-\ell+m \leq \csdd(M_0)-|D|-\ell+\ell-\lambda_{N_0}(A)
\\ &=\csdd(M)-|D|-\lambda_{N_0}(A)
.\end{align*}
Together with the symmetric inequality obtained for $C=B$, and with $N_0=M\setminus D$, we finish the proof.
\end{proof}

The following proposition is an easy corollary of \Cref{lem:startinggutsd}.

\begin{proposition}\label{cor:csxxmatrixeq}
Let $\veAm$ be a matrix over a field $\FF$ and $M=M(\ve A)$. Then $\csdd(\ve A)=\csdd(M)$ and $\cdsd(\ve A)=\cdsd(M)$.
\end{proposition}

\begin{proof}
We can use a proof analogous to the proof of \Cref{cor:csdmatrixeq} for $\csdd(\ve A)$, now referring to \Cref{lem:startinggutsd}.
For $\cdsd(\ve A)$, we apply the previous claim to the dual matroid $M^*$ and its matrix representation dual to~$\ve A$ (cf.~\Cref{obs:depthsduality}).
\end{proof}

\subsection{Algorithmic and complexity aspects}\label{sec:algoprop}

A natural question concerning structural parameters is the computational complexity of determining their value. Here, we summarise the known results in this area, highlighting both general hardness results and cases where efficient algorithms exist.

It is known that deciding whether the branch-width of a given matroid $M$ is at most $k$ is \NPh~\cite{seymour1994call}. 
For matroids represented over finite fields, a branch-width decomposition of width at most $k$ can however be computed in cubic time for every fixed $k$,
due to a result of Hliněný and Oum~\cite{HliO08}.

\begin{theorem}[\cite{HliO08}]
   There exists an \FPT algorithm that given a matroid represented over a finite field $\FF$ and an integer $k$ outputs a branch-decomposition
   of $M$ of width at most $k$, if such a decomposition exists. The algorithm runs in time $f(|\FF|, k)\cdot n^3$.
\end{theorem}

In~\cite{BriKKPS24}, Briański et al. established that computing the following matroid depth parameters is in general \NPh.
Here, the results apply for both finite and infinite fields~$\FF$.

\begin{theorem}[\cite{BriKKPS24}]\label{thm:depthNPhard}
    Let $\FF$ be any field. Given an $\FF$-represented matroid $M$ and an integer~$k$, the following problems are \NPc.
    \begin{itemize}
        \item Is the d-depth of $M$ at most $k$?
        \item Is the c-depth of $M$ at most $k$?
        \item Is the c$^*$-depth of $M$ at most $k$?
        \item Is the cd-depth of $M$ at most $k$?
        \item Is the \csddepth of $M$ at most $k$?
    \end{itemize}
\end{theorem}

From duality, the same hardness also holds for the following parameters.

\begin{corollary}
    Let $\FF$ be any field. Given an $\FF$-represented matroid $M$ and an integer $k$, the following problems are \NPc.
    \begin{itemize}
        \item Is the \dsdepth of $M$ at most $k$?
        \item Is the \cdsdepth of $M$ at most $k$?
    \end{itemize}
\end{corollary}

On the positive side, Chan et al.~\cite{ChaCKKP22} showed that if $k$ is fixed, there is a polynomial-time algorithm for computing c$^*$-depth of matroids represented over finite fields.

\begin{theorem}[\cite{ChaCKKP22}]
    \label{algo:fpt_csd}
    Let $\FF$ be a finite field. There exists an \FPT algorithm that given an $\FF$-represented matroid $M$ and an integer $k$
    decides whether the c$^*$-depth of $M$ is at most $k$ in time $f(|\FF|, k)\cdot |M|^{\mathcal{O}(1)}$. 
\end{theorem}

By duality, we also obtain that computing the d$^*$-depth of matroids represented over finite fields 
can be done in \FPT time.

\begin{corollary}
    \label{cor:dsd_fpt}
    Let $\FF$ be a finite field. There exists an \FPT algorithm that given an $\FF$-represented matroid $M$ and an integer $k$
    decides whether the d$^*$-depth of $M$ is at most $k$ in time $f(|\FF|, k)\cdot |M|^{\mathcal{O}(1)}$.
\end{corollary}

We remark that \Cref{algo:fpt_csd} is 
originally stated for contraction$^*$-depth in the sense of \Cref{def:csd-decomposition}. 
Moreover, whenever the answer is positive, 
the algorithm outputs a certifying contraction$^*$-depth decomposition 
of depth at most $k$. The algorithm underlying \Cref{algo:fpt_csd} 
is purely combinatorial and is based on dynamic programming.
As a starting point, it employs the approximation algorithm of 
Kardo\v{s} et al.~\cite{KarKLM17}, which applies 
to general matroids given by an independence oracle and computes a contraction$^*$-depth decomposition whose depth is at most exponential in the contraction$^*$-depth.

\begin{theorem}[\cite{KarKLM17}]\label{thm:csdapprox}
    There exists a polynomial algorithm that, given a matroid $M$ of contraction$^*$-depth~$d$, 
    computes a contraction$^*$-depth decomposition of depth at most $4^d$.
\end{theorem}

Subsequently, Briański et al.~\cite{BriKKPS24} showed that the problem of computing d-depth of matroids
represented over finite fields is also in \FPT.

\begin{theorem}[\cite{BriKKPS24}]
    \label{algo:fpt_dd}
    Let $\FF$ be a finite field. There exists an \FPT algorithm that given an $\FF$-represented matroid $M$ and an integer $k$
    decides whether the d-depth of $M$ is at most $k$ in time $f(|\FF|, k)\cdot|M|^{\mathcal{O}(1)}$.
\end{theorem}

In contrast to the combinatorial algorithm of \Cref{algo:fpt_csd},
the algorithm in \Cref{algo:fpt_dd} follows a model-checking approach,
relying on the fact that the property of having deletion-depth at most $d$
is definable in monadic second-order logic (MSO) directly from the definition 
(one simply uses a special existential quantifier for each level of the recursion).
Since $\bw(M)\le \dd(M)$ for every matroid $M$, and MSO-definable properties can be decided in polynomial time
on matroids of bounded branch-width which are represented over a finite field~$\FF$ \cite{Hli06}
(\cite{Hli06} does not need a branch-decomposition of $M$ on input, only an $\FF$-representation), 
this yields an \FPT algorithm parameterised by $d$ and $|\FF|$ for deciding whether $\dd(M)\le d$.

By duality between d-depth and c-depth, as a corollary, we obtain that computing c-depth of matroids represented
over finite fields can also be done in FPT time.

\begin{corollary}
    \label{cor:cd_fpt}
    Let $\FF$ be a finite field. There exists an \FPT algorithm that given an $\FF$-represented matroid $M$ and an integer $k$
    decides whether the c-depth of $M$ is at most $k$ in time $f(|\FF|, k)\cdot|M|^{\mathcal{O}(1)}$. 
\end{corollary}

Finally, the MSO formula from \Cref{algo:fpt_dd} can be routinely extended to handle a contraction of an element
equivalently to its deletion (this applies to each level of the recursion in the construction of the formula),
so we can also obtain an analogous result for the notion of cd-depth.

\begin{corollary}
    \label{cor:cdd_fpt}
    Let $\FF$ be a finite field. There exists an \FPT algorithm that given an $\FF$-represented matroid $M$ and an integer $k$
    decides whether the cd-depth of $M$ is at most $k$ in time $f(|\FF|, k)\cdot|M|^{\mathcal{O}(1)}$. 
\end{corollary}

A direct extension of the model-checking approach of \Cref{algo:fpt_dd} to our starred measures, namely to the c$^*$d-depth and the cd$^*$-depth,
is not easy due to the impossibility of defining (even just existentially) a single c$^*$- or d$^*$-transformation.
This shortcoming of monadic second-order logic of matroids can be overcome with the following alternative description of the c$^*$d-depth (which dually applies also to the cd$^*$-depth).

\begin{definition}
Let $M$ be a matroid. The \emph{gd-depth} of $M$ is defined as follows:
\begin{enumerate}[(I)]
\item gd-depth$(M)=1$ if $M$ has only one element, and
\item otherwise gd-depth$(M)= \min\{d,g\}$ where $d$ is the minimum, ranging over all $e\in E(M)$, of $1+$gd-depth$(M\setminus e)$,
and $g$ is the minimum, ranging over all bipartitions $(A,B)$ of~$E(M)$, of the value $\lambda_M(A)+\max\{$gd-depth$(M\contract A),$\,gd-depth$(M\contract B)\}$.
\label{it:gdiii}
\end{enumerate}
\end{definition}


\begin{lemma}\label{lem:gddepth}
    For every matroid $M$, the gd-depth of $M$ equals the c$^*\!$d-depth of~$M$.
\end{lemma}
\begin{proof}
We start with a proof in the `$\geq$' direction, using a structural induction on the definition of gd-depth.
If $M$ has one element, then gd-depth$(M)=1=\csdd(M)$, as required.
If gd-depth$(M)=1+$gd-depth$(M\setminus e)$, then gd-depth$(M\setminus e)\geq\csdd(M\setminus e)$ by induction and $\csdd(M)\leq1+\csdd(M\setminus e)$ by the definition,
and so gd-depth$(M)\geq1+\csdd(M\setminus e)\geq\csdd(M)$.
At last, assume that, up to symmetry, gd-depth$(M)=\lambda_M(A)+$gd-depth$(M\contract A)$ and \mbox{gd-depth}$(M)\geq\lambda_M(A)+$gd-depth$(M\contract B)$ for a bipartition $(A,B)$ of~$E(M)$.
By induction, gd-depth$(M\contract A)\geq\csdd(M\contract A)$.
By iterated application of \Cref{thm:GGWadde}, there is a sequence of $\lambda_M(A)=\lambda_M(B)$ c$^*$-transformations of the matroid $M$ which result in the direct sum of the matroids $M\contract A$ and $M\contract B$.
This, in turn, certifies by definition that $\csdd(M)\leq\lambda_M(A)+\csdd(M\contract A)$,
and we again conclude that gd-depth$(M)=\lambda_M(A)+$\mbox{gd-depth}$(M\contract A)\geq\lambda_M(A)+\csdd(M\contract A)\geq\csdd(M)$.

In the opposite `$\leq$' direction, we employ \Cref{lem:startinggutsd}.
Again, if $M$ has one element, then gd-depth$(M)=1=\csdd(M)$.
So, we may consider that $M$ has more than one element, and that there exists a tripartition $(A,B,D)$ of $E(M)$ with the properties claimed in \Cref{lem:startinggutsd}.
In particular, for each $X\in\{A,B\}$, we have $\csdd(M\setminus D\contract X)+\lambda_{M}(A)+|D| \leq \csdd(M)$.
By induction, gd-depth$(M\setminus D\contract X)\leq\csdd(M\setminus D\contract X)$.
Let $M'=M\setminus D$, and note that $\lambda_{M}(A)=\lambda_{M'}(X)$.
Since $(A,B)$ is now a bipartition of $E(M')$, from the definition we get gd-depth$(M')\leq\lambda_{M'}(X)+\max_{X\in\{A,B\}}\,$gd-depth$(M'\contract X)$.
Likewise, we get from the definition that gd-depth$(M)\leq|D|+\,$gd-depth$(M')$.
Altogether, gd-depth$(M)\leq|D|+\lambda_{M'}(X)+\max_{X\in\{A,B\}}\,$gd-depth$(M'\contract X)\leq\csdd(M)$.
\end{proof}

We may now formulate a new algorithmic result for the c$^*\!$d-depth and cd$^*$-depth as an easy corollary of \Cref{lem:gddepth}.

\begin{corollary}
    \label{cor:csdd_fpt}
    Let $\FF$ be a finite field. There exists an \FPT algorithm that given an $\FF$-represented matroid $M$ and an integer $k$
    decides whether the c$^*\!$d-depth of $M$ (resp., the cd$^*$-depth of $M$) is at most $k$ in time $f(|\FF|, k)\cdot|M|^{\mathcal{O}(1)}$. 
\end{corollary}
\begin{proof}
We follow the model-checking approach of \Cref{algo:fpt_dd} and use \Cref{lem:gddepth} in it:
for every integer $d$ we express the property that gd-depth$(M)\leq d$ in MSO logic of matroids in a way analogous to \cite{BriKKPS24}.
In the case of cd$^*$-depth, we apply the same algorithm to the dual matroid~$M^*$.
\end{proof}

\subsection{Connection to graph depth parameters}\label{sub:graphdepth}

First, we compare the tree-depth of a graph with the c-depth of its cycle matroid $M(G)$.
It is well known that graph classes of bounded tree-depth are precisely those in which all paths have bounded length.
An analogous statement holds for $2$-connected graphs and cycles~\cite[Section~6.2]{sparsity-book}.

\begin{proposition}[e.g., \cite{sparsity-book}]\label{thm:cycles-treedepth}
~
\begin{enumerate}[(a)]
\item If $\ell$ is the length of the longest path in a graph $G$, then $\lceil\log_2 \ell\rceil \le \td(G) \le \ell$.
\item If $G$ is a $2$-connected graph and $\ell$ is the length of the longest cycle in $G$, then
\[1 + \lceil\log_2 \ell\rceil \le \td(G) \le 1 + (\ell -2)^2.\]
\end{enumerate}
\end{proposition}

Combining \Cref{thm:cycles-treedepth} together with \Cref{thm:cd_circuits} immediately yields the following corollary originally stated in~\cite{DeVKO20}.

\begin{corollary}[\cite{DeVKO20}]\label{cor:td-cd}
There are functions $f$ and $g$ such that for any graph $G$, $\cd(M(G)) \le f(\td(G))$, and for any $2$-connected graph $G$, $\td(G) \le g(\cd(M(G)))$.
\end{corollary}

Note that, by \Cref{cor:xd_circuits}, ~\Cref{cor:td-cd} holds if and only if it holds when the c-depth is replaced by \csdepth.
Likewise, referring to \Cref{thm:mtreedepth-rel} we also immediately derive:

\begin{corollary}\label{cor:func-equiv-td-mtd}
There are functions $f$ and $g$ such that for any $2$-connected graph $G$, $\mtd(M(G)) \le f(\td(G))$ and $\td(G) \le g(\mtd(M(G)))$.
\end{corollary}

Observe that the requirement of $2$-connectedness in the second part of \Cref{cor:td-cd} and of \Cref{cor:func-equiv-td-mtd} is necessary because the class of all trees has unbounded tree-depth but the cycle matroid of any tree has c-depth $1$ and matroid tree-depth~$0$.

We remark that all matroids used in the proof of \Cref{pro:allvalid} are graphic (recall the ``fat cycle'' depicted in \Cref{fig:fat-cycle}), which means that the relations depicted in \Cref{fig:hasse-csdsd} are valid when restricted to graphic matroids.
However, when restricted to cycle matroids of $3$-connected graphs, all
parameters between c-depth and \csdsdepth (equivalently, branch-depth) become functionally equivalent to each other.

\begin{proposition}[{\cite[Proposition 5.19]{DeVKO20}}]\label{prop:3-connected}
There are functions $f$ and $g$ such that for any $3$-connected graph $G$ of tree-depth $t$, we have
$f(t) \le \csdsd(M(G)) \le \cd(M(G)) \le g(t)$.
\end{proposition}

We prove that the d-depth (and \dsdepth) of the cycle matroid of a $3$-connected graph $G$ is not functionally equivalent to the tree-depth of $G$.
Note that the following proposition can be easily generalised to $k$-connected graphs 
for any $k \ge 4$. 

\begin{proposition}\label{pro:K3nexample}
For each $n \ge 1$, the tree-depth of $K_{3, n}$ is at most $4$ and the d-depth of $M(K_{3,n})$ is at least $n$.
\end{proposition}
\begin{proof}
It can be easily observed that $K_{3, n}$ has tree-depth $4$, for every $n \ge 3$.
Let $M_n = M(K_{3, n})$ be a matroid.
Clearly, $\dd(M_1) = 1$.
Let $n > 1$ and observe that $M_n$ is a connected matroid.
Let $M' = M_n \setminus e$ for some $e \in E(M_n)$; observe that all choices of $e$ lead to the same matroid $M'$, up to isomorphism.
By \Cref{def:eight-parameters}, $\dd(M_n) = 1 + \dd(M')$.
Observe that $M_{n-1}$ is a restriction of $M'$, and so trivially, we obtain $\dd(M_{n-1})\leq\dd(M')$.
Hence, by induction hypothesis, $\dd(M') \ge \dd(M_{n-1}) \ge n-1$, which concludes the proof.
\end{proof}

Since matroid connectedness corresponds to graph $2$-connectedness, it is natural to compare matroid depth parameters also to the recently introduced \emph{$2$-tree-depth}~\cite{HuynhJMSW22}.
Recall that the decomposition of $G$ into blocks partitions the \emph{edges} of $G$ into maximal $2$-connected subgraphs and single edges, which corresponds to the components of the matroid $M(G)$.
We can define $2$-tree-depth as follows~\cite{hodor2025treedepth}:
\[ \td_2(G) = \begin{cases}
1 \text{ if } |V(G)| = 1,\\
\max_{i \in [k]} \td_2(B_i) \text{ if $G$ consists of blocks $B_1, \ldots, B_k$ where $k > 1$},\\
1 + \min_{v \in V(G)} \td_2(G - v) \text{ if } G \cong K_2 \text{ or  } G \text{ is $2$-connected}.
\end{cases}\]

Now we compare $2$-tree-depth with \csdsdepth and cd-depth. 

\begin{proposition}\label{prop:td2-no-csdsd}
There is no function $f$ such that for every graph $G$, $\csdsd(M(G)) \le f(\td_2(G))$.
\end{proposition}
\begin{proof}
Let $\ca C$ be a graph class such that $G \in \ca C$ iff $G$ can be obtained from a tree (with at least two vertices) by adding two universal vertices.
Clearly, each graph in $\ca C$ is $3$-connected and $\ca C$ has unbounded tree-depth, which means that the matroid class $\{M(G) \sep G \in \ca C\}$ has unbounded \csdsdepth by \Cref{prop:3-connected}.
However, each graph in $\ca C$ has $2$-tree-depth $4$, which concludes the proof. 
\end{proof}

\begin{proposition}\label{prop:td2-cdd}
For any graph $G$, $\td_2(G) \le 2\cdot \cdd(M(G))$. 
\end{proposition}
\begin{proof}
Let $M := M(G)$.
We proceed by induction on $|E(M)|$.
If $|E(M)| \le 1$, then $\cdd(M) = 1$ and $\td_2(G) \le 2$ because each connected component of $G$ contains at most two vertices.
The case when $M$ is disconnected is trivial because each component of $M$ is a block of $G$.
Hence, assume that $M$ is connected.
By \Cref{def:eight-parameters}, there is an edge $e = uv \in E(G)$ such that $k := \cdd(M) = 1 + \cdd(M(G_1))$, where $G_1 \in \{G \setminus e, G \contract e\}$.
By induction hypothesis, $\td_2(G_1) \le 2(k-1)$.
Since $G_2 := G - \{u, v\}$ is an induced subgraph of $G_1$, we obtain $\td_2(G_2) \le 2(k-1)$.
Hence, $\td_2(G) \le 2k$ as desired because adding two vertices can increase $\td_2$ by at most~$2$.
\end{proof}

We leave as an open problem whether \Cref{prop:td2-cdd} can be strengthened by replacing $\cdd$ with $\csdd$, $\cdsd$ or $\csdsd$.
The issue is that a c$^*$-transformation or a d$^*$-transformation may lead to a non-graphic matroid.
It is natural to ask whether the issue disappears if we allow only ``graphic transformations''.
The answer is positive; cd-depth can be replaced with a ``graphic'' variant of \csdsdepth, which we now define.

\begin{definition}\label{def:graphic}
Let $G$ be a graph. A graph $G'$ is a \emph{c$^{*\rm g}$-transformation} (resp. \emph{d$^{*\rm g}$-tran\-sformation}) of $G$ if there is a graph $G^+$ and an edge $e \in E(G^+)$ such that $G = G^+ \setminus e$ and $G' = G^+ \contract e$ (resp. $G = G^+ \contract e$ and $G' = G^+ \setminus e$).
The \csdsdepth of a graph $G$, denoted $\csdsd(G)$, is defined as follows:
\begin{enumerate}[(I)]
\item If $G$ has at most one edge, then $\csdsd(G) = 1$.
\item If $G$ is not $2$-connected, then $\csdsd(G) = \max\{\csdsd(G')\mid G'$ a block of $M\}$.
\item If $G$ is $2$-connected, then 
$\csdsd(G) = 1 + \min\{\csdsd(G')\mid G'$ is a c$^{*\rm g}$- or d$^{*\rm g}$-transformation of~$G\}$.
\end{enumerate}
\end{definition}

Analogously, graph variants of other starred matroid parameters can be defined.
Notice that $\csdsd(M(G)) \le \csdsd(G) \le \cdd(M(G))$ for any graph $G$.
We do not know whether $\csdsd(G) \le f(\csdsd(M(G)))$ for some function $f$.
However, we prove the following.

\begin{proposition}\label{prop:td2-csdsd}
For any graph $G$, $\td_2(G) \le 2\cdot \csdsd(G)$. 
\end{proposition}
\begin{proof}
We proceed by induction on $k + |E(G)|$, where $k = \csdsd(G)$.
If $k = 1$, then $G$ is not $2$-connected and each block of $G$ is a single edge.
Hence, $G$ is a forest, and $\td_2(G) \le 2$ as desired.
Suppose that $k > 1$.
The case when $G$ is not $2$-connected is trivial because each block of $G$ has fewer edges than $G$, so assume that $G$ is $2$-connected.

By \Cref{def:graphic}, there is a graph $G_1$ that is a c$^{*\rm g}$- or d$^{*\rm g}$-transformation of $G$ such that $\csdsd(G_1) = k - 1$.
By induction hypothesis, $\td_2(G_1) \le 2(k-1)$.
Let $G^+$ be the graph such that for some edge $e = uv \in E(G^+)$, $G = G^+ \setminus e$ and $G_1 = G^+ \contract e$, or $G = G^+ \contract e$ and $G_1 = G^+ \setminus e$.
If $G_1$ is a c$^{*\rm g}$-transformation of $G$, let $G_2 := G - \{u, v\}$, and if $G_1$ is a d$^{*\rm g}$-transformation of $G$, let $G_2 := G_1 - \{u, v\}$.
In both cases, $G_2$ is an induced subgraph of both $G$ and $G_1$.
Hence, $\td_2(G_2) \le 2(k-1)$, and since $|V(G) \setminus V(G_2)| \le 2$, we obtain $\td_2(G) \le 2k$ as desired.
\end{proof}

\section{Closure properties of depth measures}\label{sec:closurede}

\subsection{\boldmath C$^*$-depth}

In this section, we return to \Cref{thm:contrclosure} of \cite{BriKL25}, which beautifully relates 
the contraction$^*$-depth parameter (\Cref{def:csd-decomposition}) to the restriction closure of c-depth.
Using the tools developed in this paper, we can easily (and independently of \Cref{thm:contrclosure} and \cite{BriKL25})
derive an analogous restriction-closure result for the c$^*$-depth:

\begin{theorem}\label{thm:csclosure}
Let $M$ be a matroid and $\ell$ denote the minimum c-depth of a matroid $M'$ that contains $M$ as a restriction.
Then~$\ell=\csd(M)$.
\end{theorem}

The proof will use \Cref{lem:startingguts} and the following straightforward technical lemma.
We first recall the notion a relatively free extension of a matroid $M$ by element $e$ in a bispan $(A,B)$,
and we shortly write `a relatively free extension of $M$' if particular $(A,B)$ and $e$ are not important.

\begin{lemma}\label{lem:commutesame}
Let $M_1$ be a matroid and $M_2=M_1/Y$ where $Y\subseteq E(M_1)$.
Assume that a matroid $M_2'$ is a relatively free extension of $M_2$ and $E(M_2')\cap Y=\emptyset$.
Then there exists a matroid $M_1'$ which is a relatively free extension of $M_1$ such that $M_2'=M_1'/Y$.
\end{lemma}

\begin{proof}
Let $M_2'$ be a relatively free extension of~$M_2$ by~$e$ in a bispan $(A_0,B_0)$.
If the bispan $(A_0,B_0)$ is not connected, then $e$ is a loop and the lemma is trivial.
Otherwise, let $(A,B)$ be a bipartition of $E(M_2)$ such that $A_0\subseteq A\subseteq\cl_{M_2}(A_0)$ and $B_0\subseteq B\subseteq\cl_{M_2}(B_0)$.
Then, in particular, $\rnk_{M_2}(A)+\rnk_{M_2}(B)\geq \rnk_{M_2}(A_0)+\rnk_{M_2}(B_0)>\rnk(M_2)$, 
and so $\rnk_{M_1}(A\cup Y)+\rnk_{M_1}(B)\geq\rnk_{M_2}(A)+\rnk_{M_1}(Y)+\rnk_{M_2}(B)>\rnk(M_2)+\rnk_{M_1}(Y)=\rnk(M_1)$
which means that the bispan $(A\cup Y,B)$ is connected in~$M_1$.
Using \Cref{thm:GGWadde}, we define $M_1'$ as a relatively free extension of $M_1$ by $e$ in $(A\cup Y,B)$.

By the  definition of a relatively free extension, for every $Z\subseteq E(M_2)$ we have 
$e\in\cl_{M_2'}(Z)$ if and only if $\rnk_{M_2}(A\cup Z)+\rnk_{M_2}(B\cup Z)=\rnk_{M_2}(Z)+\rnk(M_2)$.
Since $M_2=M_1/Y$, we have $\rnk_{M_2}(A\cup Z)=\rnk_{M_1}(A\cup Y\cup Z)-\rnk_{M_1}(Y)$,
$\rnk_{M_2}(B\cup Z)=\rnk_{M_1}(B\cup Y\cup Z)-\rnk_{M_1}(Y)$ and $\rnk_{M_2}(Z)=\rnk_{M_1}(Y\cup Z)-\rnk_{M_1}(Y)$, $\rnk(M_2)=\rnk(M_1)-\rnk_{M_1}(Y)$.
Hence,
\begin{align*} e\in\cl_{M_2'}(Z) \quad\iff\quad
\rnk_{M_1}(A\cup Y\cup Z)+\rnk_{M_1}(B\cup Y\cup Z) =& \rnk_{M_1}(Y\cup Z)+\rnk(M_1)
.\end{align*}

The latter equality, again by the  definition of a relatively free extension, is equivalent to $e\in\cl_{M_1'}(Y\cup Z)$.
This is if and only if $e\in\cl_{M_1'/Y}(Z)$, and hence $M_1'/Y=M_2'$ is proved.
\end{proof}

\begin{proof}[Proof of \Cref{thm:csclosure}]
The inequality $\csd(M)\leq\ell$ has been proved in \Cref{lem:extcd-to-csd}.
In the opposite direction, we are going to prove the existence of a matroid $M'$ such that $M$ is a restriction of $M'$ and $\cd(M')\leq\csd(M)$,
by structural induction on the definition of $\csd(M)$.
Moreover, we will maintain an invariant that $M'$ is obtained from $M$ by a sequence of relatively free extensions.

If $|E(M)|=1$, then $M'=M$ and $\cd(M')=\csd(M)=1$.
Otherwise, by \Cref{lem:startingguts}, we get a bipartition $(A,B)$ of $E(M)$ such that $\>\max\{\csd(M\contract A),\csd(M\contract B)\}\leq\csd(M)-\lambda_{M}(A)$.
If $\lambda_{M}(A)=0$, that is, if $M$ is disconnected, then we simply take the components $M_1,\ldots,M_k$, $k\geq2$, of $M$
and inductively construct matroids $M_i'$, $i\in[k]$, such that $\cd(M_i')=\csd(M_i)$.
Then $M'$ is the direct sum of $M_1',\ldots,M_k'$ and $\cd(M')=\csd(M)$.

We further assume $a=\lambda_{M}(A)>0$, and define a sequence of matroids $N_0=M$ and $N_1,\ldots,N_a$ inductively as follows.
For $i=1,\ldots,a$, let $N^+_i$ be obtained, using \Cref{thm:GGWadde}, as a relatively free extension of $N_{i-1}$ by $f_i$ in $(A,B)$,
and let $N_i=N^+_i\contract f_i$.
Since $f_i\in\cl_{N^+_i}(A)\cap\cl_{N^+_i}(B)$, we get $\lambda_{N_i}(A)=\lambda_{N_{i-1}}(A)-1$, and so $\lambda_{N_i}(A)=a-i$ by induction on~$i$.

Again by induction on~$i=1,\ldots,a$, we get that there is a matroid $N_i'$ on $E(N_i')=E(M)\cup\{f_1,\ldots,f_i\}$, 
obtained by a sequence of relatively free extensions from $M$, such that $N_i=N_i'\contract\{f_1,\ldots,f_i\}$;
this is trivial for $i=1$ and follows straightforwardly from \Cref{lem:commutesame} for~$i>1$.
Let, shortly, $N=N_a'$ and $Z=\{f_1,\ldots,f_a\}$. 

Since $N_a=N\contract Z$, $\>\lambda_{N_a}(A)=a-a=0$ and $|Z|=a=\lambda_{M}(A)\leq\lambda_{N}(A)$, 
we necessarily have $\lambda_{N}(A)=a$ and $Z\subseteq\cl_N(A)\cap\cl_N(B)$.
Therefore, $N_a\contract A=N\contract(Z\cup A)=N\contract A\setminus Z=M\contract A$, and symmetrically $N_a\contract B=M\contract B$,
which means that $N_a$ can be written as a direct sum of $M\contract A$ and $M\contract B$.
So, by \Cref{lem:startingguts}, $\csd(N_a)\leq\csd(M)-a$.

By the induction assumption, we get a matroid $M''$ obtained from $N_a$ by a sequence of relatively free extensions, such that $\cd(M'')\leq\csd(M)-a$.
Let this sequence construct, in order, matroids $M''_m=N_a$ and $M''_{m-1},\ldots,M''_0$ where $m=|E(M'')-E(N_a)|$ and~$M''_0=M''$,
and for $i\in[m]$, $M''_{i-1}$ is a relatively free extension of $M''_i$ by an element~$g_i$.
We define a sequence of matroids $M'_0=N$ and $M'_1,\ldots,M'_m$ where $M'_i$ for $i\in[m]$ is obtained by invoking \Cref{lem:commutesame}
onto the matroids $M'_{i-1}$ and $M''_{i-1}$ and the set $Y=Z\cup\{g_1,\ldots,g_{i-1}\}$, 
considering a relatively free extension $M''_i$ of $M''_{i-1}$ by~$g_i$.
(This makes $M'_i$ a relatively free extension of $M'_{i-1}$ by~$g_i$.)

Hence, we construct a matroid $M'=M'_m$ which results by a sequence of $a+m$ relatively free extensions from $M$.
Moreover, $M'\contract Z=M''$ by \Cref{lem:commutesame}, and so $\cd(M')\leq\cd(M'')+|Z|=\cd(M'')+a$ by \Cref{def:eight-parameters}.
Together with $\cd(M'')\leq\csd(M)-a$, we conclude that  $\cd(M')\leq\csd(M)$ and the proof is finished.
\end{proof}

Applying things in the dual matroid, we also immediately obtain:
\begin{corollary}\label{cor:dsclosure}
Let $M$ be a matroid and $\ell$ denote the minimum d-depth of a matroid $M'$ such that $M=M'\contract X$ for some~$X\subseteq E(M')$.
Then~$\ell=\dsd(M)$.
\qed\end{corollary}

\subsection{\boldmath Contraction$^*$-depth}\label{sub:contrstartproof}

Our results actually also give a relatively simple proof of \Cref{thm:contrclosure}, thus providing an independent alternative to the lengthy proof in \cite{BriKL25}.
In this respect, we remark that, although we rely on \Cref{thm:GGWadde} from \cite{DBLP:journals/siamdm/GeelenGW06}, too, its proof in 
\cite{DBLP:journals/siamdm/GeelenGW06} takes only about one page of elementary arguments.

We recall \Cref{def:csd-general} of a contraction$^*$-depth decomposition $(T,f)$, and especially the notation
$T_X$ meaning the subtree from the root of $T$ to all leaves $f(x)$ where~$x\in X$.

\begin{lemma}\label{lem:kklm-to-csd}
Let $M$ be a matroid and $k$ denote the contraction$^*$-depth of $M$. 
Then~$\csd(M)\leq k+1$.
\end{lemma}
\begin{proof}
Let $(T,f)$ be a contraction$^*$-depth decomposition of $M$ of minimum height.
We are going to prove, by induction on the size of $M$, that $\csd(M)\leq\hght(T)=k+1$.
Note that if the root of $T$ has more than one child, and $A\subseteq E(M)$ is the set of elements mapped by $f$ 
to the descendants of one child of the root and $B=E(M)-A$,
then $\rnk_M(A)+\rnk_M(B)\leq|E(T_A)|+|E(T_B)|\leq|E(T)|=\rnk(M)$ by \Cref{def:csd-decomposition}, and hence $\lambda_M(A)=0$ and $M$ is disconnected.

So, whenever $M$ is not connected, we choose a bipartition $(A,B)$ of $E(M)$ such that $\lambda_M(A)=\rnk_M(A)+\rnk_M(B)-\rnk(M)=0$.
This implies $\rnk_M(A)=|E(T_A)|$ and $\rnk_M(B)=|E(T_B)|$, and so the subtrees $T_A$ and $T_B$ give us
contraction$^*$-depth decompositions of $M_1=M\restriction\!A$ and $M_2=M\restriction\!B$, respectively.
By induction assumption, we have $\csd(M_1)\leq\hght(T_A)\leq\hght(T)$ and $\csd(M_2)\leq\hght(T_B)\leq\hght(T)$, 
and since $M$ is a direct sum of $M_1$ and $M_2$ in this case, we conclude $\csd(M)=\max\{\csd(M_1),\csd(M_2)\}\leq\hght(T)$.

If $M$ is connected and $T$ is a path with one end in the root, then $k=\hght(T)-1=|E(T)|=\rnk(M)$, and we trivially have $\csd(M)\leq\rnk(M)=\hght(T)-1$.

Otherwise, the root has only one child, and we pick a node $t\in V(T)$ closest to the root which has at least two children.
Let $A\subseteq E(M)$ be the set of elements mapped by $f$ to one of the subtrees rooted at $t$ and $B=E(M)-A$.
We have $\lambda_M(A)>0$ since $M$ is connected.
Using \Cref{thm:GGWadde}, we obtain a matroid $M^+$ as a relatively free extension of $M$ by an element $e$ in $(A,B)$,
and set~$M'=M^+\contract e$. So, by \Cref{def:eight-parameters},~$\csd(M)\leq\csd(M')+1$.
We construct a tree $T'$ from $T$ by removing the root of $T$ (so that its child becomes the new root).
We claim that $(T',f)$ is a valid contraction$^*$-depth decomposition of $M'$.
Hence, by induction, $\csd(M')\leq\hght(T')=\hght(T)-1$ and $\csd(M)\leq\hght(T)$.

\smallskip
It remains to prove that $(T',f)$ is a contraction$^*$-depth decomposition of $M'$ according to \Cref{def:csd-decomposition}.
First, easily, $\rnk(M')=\rnk(M)-1=|E(T)|-1=|E(T')|$.
Second, we need that for every $X\subseteq E(M')$, $\rnk_{M'}(X)\leq|E(T'_X)|=|E(T_X)|-1$.
This trivially holds if $\rnk_{M}(X)<|E(T_X)|$, and so we assume $\rnk_{M}(X)=|E(T_X)|$ and focus more on the decomposition $(T,f)$ of~$M$.
In particular, we easily derive
\begin{align*}
|E(T_{A\cup X})|+|E(T_{B\cup X})|=|E(T)|+|E(T_X)| = \rnk(M)+\rnk_{M}(X)
,\end{align*}
and since $\rnk_M(A\cup X)\leq|E(T_{A\cup X})|$ and $\rnk_M(B\cup X)\leq|E(T_{B\cup X})|$ by \Cref{def:csd-decomposition}, and 
\begin{align}\label{eq:modularx}
\rnk_M(A\cup X)+\rnk_M(B\cup X) \geq \rnk(M)+\rnk_{M}(X)
\end{align}
by submodularity, we actually must have equality in \eqref{eq:modularx}. Thus, $(A\cup X,B\cup X)$ is a modular pair.

Consequently, by the definition of a relatively free extension, $e\in\cl_M(X)$.
Then, since $M'=M\contract e$, we get $\rnk_{M'}(X)=\rnk_{M}(X)-1\leq|E(T_X)|-1=|E(T'_X)|$, and the proof is finished.
\end{proof}

Now we simply combine the previous lemmas:
\begin{proof}[Alternative proof of \Cref{thm:contrclosure}]
If $M$ is of positive rank and consists of only loops and coloops, in which case $k=1$, we trivially get~$\cd(M)=1=k=\ell$.
Otherwise, we have that $k=\csd(M)-1$;
the inequality $k\leq\csd(M)-1$ is proved in \Cref{lem:csd-to-kklm} and $k\geq\csd(M)-1$ in \Cref{lem:kklm-to-csd}.
Finally, by \Cref{thm:csclosure}, we get a matroid $M'$ such that $\cd(M')=\ell=\csd(M)=k+1$.
\end{proof}

\section{Summary and Open Problems}

We have introduced a unifying framework for recursive depth measures of matroids that captures several previously studied depth notions, 
including those that were previously defined only for matrices (that is, for particular vector representations of matroids).
The contribution of this framework is two-fold.
First, we have provided abstract matroid definitions for concepts previously tied to particular vector representations,
and we have proved their basic structural properties and verified that the new abstract view exactly coincides with their previous view based on handling vector representations.

Second, our framework also
yields a few new depth measures, not previously considered in the literature,
and we have proved functional equivalence of measures within our framework to two other natural, decomposition based, depth measures
-- branch-depth and matroid tree-depth.
We have also clarified all mutual relationships and functional equivalences between the measures of our framework,
and studied some 
computational aspects of these parameters.

\medskip
Several natural questions related to the parameters remain open. 
In particular, the computational complexity of evaluating some of them is not yet known -- for example,
in the case of matroid branch-depth and \csdsdepth (cf.~\Cref{thm:depthNPhard}). 
We conjecture that both problems are computationally hard.

\begin{conjecture}\label{conj:bdhard}
Computing exactly the branch-depth and the \csdsdepth of matroids is \NPh.
\end{conjecture}

In algorithmic applications of structural measures, one often does not need to compute the measure (and preferably, also a decomposition)
exactly; a functional approximation is often enough, at least in theory.
To this end, \Cref{thm:csdapprox} shows how to compute an approximate contraction$^*$-depth decomposition,
this time without any parameter dependence and for all matroids given by an oracle.
This result implies analogous approximations for c$^*$-depth and matroid tree-depth, and, dually, for d$^*$-depth.

What about other depth measures, in particular, branch-depth?
This question is also interesting in connection with the very recent breakthrough algorithm
of Korhonen and Oum~\cite{korhonen2026branchwidthconnectivityfunctionsfixedparameter} for branch-width.

\begin{problem}\label{prb:bdapprox}
    Find a computable function $f$ and a polytime algorithm that, for a matroid $M$ given by an independence oracle, finds
    a branch-depth decomposition of $M$ of width and depth at most $f(\bd(M))$.
\end{problem}

Closely related to the task of computing the value of a depth measure is the problem of finding the ``obstructions''
for each value of the measure. That is, for a parameter $k$, to determine the set of matroids which are restriction- or contraction- or minor-minimal
with respect to having the considered depth measure strictly greater than $k$.

For instance, DeVos, Kwon, and Oum~\cite{DeVKO20} proved that, for every finite field $\mathbb{F}$,
every class of matroids representable over $\FF$ of bounded c-depth is well-quasi-ordered with respect to restriction.
This implies that the set of restriction-minimal matroids of c-depth greater than~$k$
(i.e., the set of obstructions to having c-depth at most~$k$) is finite for each $k$, albeit without giving an explicit size bound.
Gajarsk\'y, Pek{\'{a}}rkov{\'{a}} and Pilipczuk~\cite{GajPP25} later re-proved an equivalent result with an explicit upper
bound on the size of such obstructions in terms of $k$ and $|\mathbb{F}|$.
Again, this result extends to the contraction-minimal obstructions for d$^*$-depth via duality.
It would be interesting to prove an analogous result concerning, for example, the minor-minimal obstructions for c$^*$d-depth.

\medskip
Finally, we would like to return to the closure properties of depth measures, as studied in \Cref{sec:closurede}.
In \cite{BriKL25} (\Cref{thm:contrclosure}), it is proved that the contraction$^*$-depth of a matroid $M$ is equal to the restriction closure of its c-depth,
modulo the technical difference by $1$ caused by our adjusted definition of c-depth.
In \Cref{thm:csclosure}, we prove that the c$^*$-depth of a matroid $M$ is equal to the restriction closure of its c-depth, that is, to the minimum c-depth 
of a matroid $M'$ that contains $M$ as a restriction.
Analogously, \cite{DBLP:journals/corr/abs-2402-16215prep} proves that the branch-depth of a \emph{representable} matroid $M$ is functionally related to the minor closure of its cd-depth,
that is, to the minimum cd-depth of a matroid $M'$ that contains $M$ as a minor.
However, \cite{DBLP:journals/corr/abs-2402-16215prep} leaves open the question whether the same statement holds for all matroids.

In the setting of our paper, this naturally brings a new question, perhaps a bit easier than the one from~\cite{DBLP:journals/corr/abs-2402-16215prep}:
\begin{problem}
Is it true that, for any matroid $M$, the c$^*$d-depth of $M$ is functionally related (and perhaps equal) to the minimum cd-depth of a matroid $M'$ that contains $M$ as a restriction?
\end{problem}

There is yet another problem left open in our paper, not directly related to closure properties, but seemingly affected by the same kind of difficulties as those encountered in \cite{DBLP:journals/corr/abs-2402-16215prep} for the closure property of branch-depth in general matroids:
\begin{problem}
Is it true, that for any matrix $\ve A$ over any field, $\csdsd(\ve A)=\csdsd(M(\ve A))$?
Or, in a weaker sense, is $\csdsd(\ve A)$ functionally related to $\csdsd(M(\ve A))$?
\end{problem}
We have, so far unverified, indications that exact equality $\csdsd(\ve A)=\csdsd(M(\ve A))$ may not always hold, but the weaker question of functional relation is completely open to our knowledge.

\subparagraph*{Acknowledgements.}
We acknowledge a hint from ChatGPT on a construction of the counterexample matroid $L_k$ in \Cref{pro:fat-cycle-dcycle}.
However, the whole proof of \Cref{pro:fat-cycle-dcycle}, like all other text, statements and proofs in this paper,
have been written by humans without assistence of AI.

\bibliography{bibliography}

\end{document}